\documentclass[final,reqno]{siamltex}
\usepackage{latexsym,amsmath,amssymb,amsfonts,mathrsfs}
\usepackage{epsf,graphicx,epsfig,color,cite,cases,bm}
\usepackage{subfigure,graphics,multirow,marginnote,graphicx}
\usepackage{enumerate,dsfont}

\usepackage{color}
\definecolor{HW}{rgb}{0,0,0}
\newcommand{\R}{{\mat R}}

\newcommand{\Sp}{{\mathbb S}}
\newcommand{\ds}{\displaystyle}
\newcommand{\no}{\nonumber}
\newcommand{\be}{\begin{eqnarray}}
\newcommand{\ben}{\begin{eqnarray*}}
\newcommand{\en}{\end{eqnarray}}
\newcommand{\enn}{\end{eqnarray*}}
\newcommand{\ba}{\backslash}
\newcommand{\pa}{\partial}

\newcommand{\ov}{\overline}
\newcommand{\I}{{\rm Im}}
\newcommand{\Rt}{{\rm Re}}

\newcommand{\wid}{\widetilde}

\newcommand{\mat}{\mathbb}
\newcommand{\se}{\setminus}

\newcommand{\la}{\lambda}

\newtheorem{remark}[theorem]{Remark}
\newtheorem{algorithm}{Algorithm}[section]

\begin{document}
\renewcommand{\theequation}{\arabic{section}.\arabic{equation}}
\title{\bf An approximate factorization method for inverse acoustic scattering with phaseless {\color{HW}total-field} data
}
\author{Bo Zhang\thanks{LSEC, NCMIS and Academy of Mathematics and Systems Science, Chinese Academy of
Sciences, Beijing 100190, China and School of Mathematical Sciences, University of Chinese
Academy of Sciences, Beijing 100049, China ({\tt b.zhang@amt.ac.cn})}
\and Haiwen Zhang\thanks{Corresponding author. NCMIS and Academy of Mathematics and Systems Science, Chinese Academy of Sciences,
Beijing 100190, China 
and Institute for Numerical and Applied Mathematics,
University of G\"{o}ttingen, Lotzestr. 16-18, 37083 G\"{o}ttingen, Germany
({\tt zhanghaiwen@amss.ac.cn})}
}
\date{}

\maketitle

%\vspace{.2in}

\begin{abstract}
This paper is concerned with the inverse acoustic scattering problem with phaseless {\color{HW}total-field}
data at a fixed frequency.
An approximate factorization method is developed to numerically reconstruct both the location and
shape of the unknown scatterer from the phaseless {\color{HW}total-field} data generated by incident plane
waves at a fixed frequency and measured on the circle $\pa B_R$ with a sufficiently large radius $R$.
The theoretical analysis of our method is based on the asymptotic property in the operator norm from
$H^{1/2}(\Sp^1)$ to $H^{-1/2}(\Sp^1)$ of the phaseless {\color{HW}total-field} operator defined in terms of
the phaseless {\color{HW}total-field} data measured on $\pa B_R$ with large enough $R$, where $H^s(\Sp^1)$ is a
Sobolev space on the unit circle $\Sp^1$ for real number $s$,
together with the factorization of a modified far-field operator.
The asymptotic property of the phaseless {\color{HW}total-field} operator is also established in this paper
with the theory of oscillatory integrals.
The unknown scatterer can be either an impenetrable obstacle of sound-soft, sound-hard or impedance type
or an inhomogeneous medium with a compact support, and the proposed inversion algorithm does not need to
know the boundary condition of the unknown obstacle in advance.
Numerical examples are also carried out to demonstrate the effectiveness of our inversion method.
To the best of our knowledge, it is the first attempt to develop a factorization type method
for inverse scattering problems with phaseless data.


\end{abstract}

\pagestyle{myheadings}
\thispagestyle{plain}
\markboth{Bo Zhang and Haiwen Zhang}
{Factorization method for phaseless inverse problem}

\section{Introduction}\label{sec1}
\setcounter{equation}{0}

Inverse scattering with phased data (i.e., data with phase information) has been widely studied
mathematically and numerically over the past decades due to its significant applications in such diverse
scientific areas as radar and sonar detection, remote sensing, geophysics, medical imaging and nondestructive
testing (see, e.g., \cite{CC14,CK11,CK13,K11,KG08} for a comprehensive overview).
However, in many practical applications, it is much harder to obtain data with accurate
phase information compared with only measuring the modulus or intensity of the data
(see, e.g., \cite{Chen18,CLS15,MH17,MDS93} and the references quoted there).
Therefore, it is often desirable to study inverse scattering problems with phaseless data
(i.e., data without phase information).

Many optimization and iteration algorithms have been proposed for solving inverse scattering problems
with phaseless data (see, e.g., \cite{AHN18,BLL13,BZ16,DZG19,DLL19,I07,IK10,IK11,KR97,LZL09,PZCY11,ZZ17,ZZ17b}).
Optimization and iteration algorithms can achieve an accurate reconstruction of the unknown scatterers.
However, this type of algorithms is time-consuming and needs to know the boundary conditions of the
unknown scatterers in advance. To reduce the computational cost, non-iterative algorithms
have recently attracted more and more attention in inverse scattering problems
(see, e.g., \cite{CC14,KG08,P06}).
For inverse scattering problems with phaseless data, non-iterative algorithms have also been studied recently.
In \cite{LL15}, a non-iterative algorithm was proposed to reconstruct a polyhedral sound-soft or sound-hard
obstacle from a few high frequency phaseless backscattering far-field measurements associated with incident
plane waves, where the exterior unit normal vector of each side/face of the obstacle is determined first
with suitably chosen incident directions and the location of the obstacle is then determined with
a few phased far-field data.
This method has been extended to the case of inverse electromagnetic scattering in \cite{LLW17}.
In \cite{CH16,CH2017,CHF2017}, a direct imaging method was proposed to reconstruct scattering obstacles
from acoustic and electromagnetic phaseless {\color{HW}total-field} data, based on the reverse time migration technique.
%A direct imaging algorithm was also introduced in \cite{XZZ19} to recover a locally Dirichlet rough surface
%from the phaseless total-field data generated by plane waves and measured on the upper part of a large circle.
A direct imaging method was developed in \cite{ZZ18} to recover scattering obstacles from acoustic phaseless
far-field data corresponding to infinitely many sets of superpositions of two plane waves with a fixed frequency
as the incident fields. Recently, two direct sampling algorithms were given to reconstruct acoustic obstacles
in \cite{JLZ19a} and acoustic sources in \cite{JLZ19b} from phaseless far-field data generated with incident
plane waves, by adding a reference point scatterer into the scattering system.
In \cite{JLZ19c}, two direct sampling algorithms were introduced to recover acoustic obstacles from
phaseless far-field measurements corresponding to
superpositions of plane waves and point sources as the incident fields, where the point sources have fixed
source location with at most three different scattering strengths.
In \cite{ZGLL18}, a non-iterative algorithm was proposed to recover acoustic sources from multi-frequency
phaseless near-field data, where the phase information of the measured data is first recovered with
the reference point source technique (that is, adding certain reference point sources into the scattering system)
and the acoustic sources are then reconstructed with the Fourier method.
{\color{HW}
Recently, numerical algorithms have been proposed in \cite{KKNNBA18,KNP16,KNN19} to recover the
refractive indices of unknown scatterers from multi-frequency phaseless scattered-field or total-field data.
Inversion algorithms have also been developed in \cite{MNP16,MNPT17,NMP15} for array imaging problems
with phaseless data.}
On the other hand, uniqueness and stability results have also been obtained for inverse scattering with
phaseless data (see, e.g., \cite{Kli14,Kli17,KR16,MH17,N15,RY18} for the case of near-field measurements
and \cite{ACZ16,JLZ19a,JLZ19b,JLZ19c,LZ10,M76,N16,XZZ18,XZZ18b,ZG18} for the case of far-field measurements).

In this paper, we consider inverse acoustic scattering with phaseless {\color{HW}total-field} data associated with
incident plane waves. For simplicity, we restrict our attention to the 2D case.
Precisely, our inverse problem is to reconstruct a unknown scatterer from the phaseless
{\color{HW}total-field} data $|u(x,d)|$ for $x\in\pa B_R$ and $d\in\Sp^1$,
where $\pa B_R$ is a circle of radius $R$ and centered at the origin enclosing the unknown obstacle,
$\Sp^1$ is the unit circle, $d\in\Sp^1$ is the direction of incident plane wave
and $u$ is the total field which is the sum of the incident and scattered fields.
Our purpose here is to develop a numerical algorithm based on the factorization method to solve this
inverse problem.
The classical factorization method was first proposed by Kirsch in \cite{K98}, where a necessary and
sufficient criterion was established to characterize both the location and shape of the obstacle
by using the spectral system of the far-field operator defined by the far-field pattern associated
with the incident plane waves. Moreover, this method can be implemented as a non-iterative
algorithm for the inverse scattering problem which is very fast in the computations and does not need
to know the boundary conditions of the unknown obstacles in advance.
Therefore, the factorization method has been widely applied to many kinds of inverse scattering
problems so far. In particular, the factorization method was extended to the case of near-field
measurements in \cite{HYZZ14} to reconstruct the unknown obstacle with the aid of the spectral system
of the near-field operator defined by the near-field data measured over a circle or sphere,
associated with the incident point sources.
We refer to the monograph \cite{KG08} and the references therein for a comprehensive overview
of the factorization method.

Note that, in order to establish the necessary and sufficient criterion on the characterization of
the unknown obstacles, the key step of the classical factorization method is to prove that
the constructed far-field operator or near-field operator defined by the measured data satisfies
the Range Identity (see \cite[Theorem 2.15]{KG08} for the original version and \cite[Theorem 3.2]{BH13}
and \cite[Theorem 1.1]{KL14} for the modified version).
However, for the inverse problem under consideration, it is difficult to find a suitable operator
defined by the measured phaseless {\color{HW}total-field} data which satisfies the Range Identity.
Thus, in this paper, we propose a modified factorization method, which is called the approximate
factorization method, to numerically reconstruct the unknown scatterer from the phaseless
{\color{HW}total-field} data. In doing so, an essential role is played by the asymptotic property in the linear space
$\mathcal{L}(H^{1/2}(\Sp^1), H^{-1/2}(\Sp^1))$ of the phaseless {\color{HW}total-field} operator defined in terms
of the phaseless {\color{HW}total-field} data measured on the circle $\pa B_R$ with large enough $R$.
This asymptotic property is established, in this paper, by making use of the asymptotic properties of
the scattered field and results from the theory of oscillatory integrals (see Theorem \ref{thm1}).
In particular, utilizing this asymptotic result and constructing a modified phaseless {\color{HW}total-field} operator
$\wid{N}^{PW}_{R}$ defined by the measured data $|u(x,d)|$ with $x\in\pa B_R,\;d\in\Sp^1$
and a modified far-field operator $\wid{F}$ defined by the far-field pattern $u^\infty(\hat{x},d)$
of the scattered field with $\hat{x},d\in\Sp^1$ (see the formulas (\ref{eq37}) and (\ref{eq38}) below),
we can prove that the operators $\ds(\wid{N}^{PW}_{R})_{\#}:=|\Rt(\wid{N}^{PW}_{R})|+|\I(\wid{N}^{PW}_{R})|$
and $\wid{F}_{\#}:=|\Rt(\wid{F})|+|\I(\wid{F})|$ satisfy the asymptotic property
\ben
\|(\wid{N}^{PW}_{R})_{\#}-(1/\sqrt{8 k\pi R})\wid{F}_\#\|_{L^2(\Sp^1)\rightarrow L^2(\Sp^1)}
=O(1/R^{\alpha})
\enn
for any fixed $\alpha\in (1/2,1)$ as the radius $R\rightarrow+\infty$ (see Remark \ref{rem2}).
This means that the leading order term of the operator $\ds(\wid{N}^{PW}_{R})_{\#}$
in the linear space $\mathcal{L}(L^2(\Sp^1), L^2(\Sp^1))$ is
$\ds{(1/\sqrt{8 k\pi R})}\wid{F}_\#$ as $R\rightarrow+\infty$.
Note that ${(1/\sqrt{8 k\pi R})}$ is a constant for fixed $R$.
On the other hand, we can prove that the operator $\wid{F}_{\#}$ has a factorization satisfying
the Range Identity in \cite[Theorem 1.1]{KL14} and thus the unknown obstacle can be recovered from
the spectral system of $\wid{F}_{\#}$ (see Theorem \ref{thm3} below).
Thus, it is expected that
the unknown obstacle can be approximately recovered from the spectral system of
$(\wid{N}^{PW}_{R})_{\#}$ if $R$ is sufficiently large.
Based on this, a numerical algorithm is proposed to reconstruct both the location and shape of the
unknown scatterer from the phaseless {\color{HW}total-field} data.
Numerical examples are also carried out to demonstrate the effectiveness of the inversion algorithm.
It should be remarked that an approximate factorization or asymptotic factorization method has also
been studied for inverse scattering problems with phased data (see \cite{AGH07,G08,GH09,QYZ17,QZ19}).
To the best of our knowledge, the present paper is the first attempt to employ the idea of the
factorization method in inverse scattering problems with phaseless data.

The rest part of this paper is organised as follows. In Section \ref{sec1+}, we present the forward and
inverse scattering problems considered. In Section \ref{sec2}, we study the asymptotic property in
the linear space $\mathcal{L}(H^{1/2}(\Sp^1), H^{-1/2}(\Sp^1))$ of the phaseless {\color{HW}total-field} operator
defined in terms of the phaseless {\color{HW}total-field} data measured on the circle $\pa B_R$ with large enough $R$
(see Theorem \ref{thm1} below), which plays an essential role in the theoretical analysis of the
approximate factorization method given in Section \ref{sec3} for the inverse problem under consideration.
The numerical implementation of our method is presented in Section \ref{sec4}.
Numerical experiments are carried out in Section \ref{sec5} to illustrate the effectiveness of the
inversion algorithm. Some conclusions are given in Section \ref{sec6}.

\section{The forward and inverse scattering problems}\label{sec1+}
\setcounter{equation}{0}

We now present the forward and inverse scattering problems considered in this paper.
For simplicity, we restrict our attention to the 2D case.
However, our analysis can be easily extended to the 3D case.
Let the obstacle $D$ be an open and bounded domain in $\R^2$ with $C^2$-boundary $\pa D$ such that
the exterior $\R^2\ba\ov{D}$ of $\ov{D}$ is connected. Given the incident field $u^i$, the total field
$u=u^i+u^s$ is the sum of the incident field $u^i$ and the scattered field $u^s$.
If $D$ is an impenetrable obstacle, then the scattering problem by the obstacle $D$ is modeled as follows:
%scattered field $u^s$ satisfies the exterior boundary value problem
\begin{align}\label{eq1}
\qquad\qquad\qquad\qquad
\Delta u^s+k^2 u^s=&\;0&&\quad{\rm in}\;\;\R^2\se\ov{D},
\qquad\qquad\qquad\\ \label{eq2}
{\mathscr B}u^s=&\;f&&\quad{\rm on}\;\;\pa D,\\ \label{eq3}
\lim\limits_{r\rightarrow\infty}r^{\frac{1}{2}}\left(\frac{\pa u^s}{\pa r}-iku^s\right)=&\;0,&&\quad r=|x|,
\end{align}
where $u^s:=u-u^i$ is the scattered field, the boundary value $f:=-{\mathscr B}u^i$, $k=\omega/c>0$ is
the wave number, and $\omega$ and $c$ are the wave frequency and speed in the homogeneous medium in
$\R^2\ba\ov{D}$.
Here, the equation (\ref{eq1}) is called the Helmholtz equation, and the condition (\ref{eq3}) is the
well-known Sommerfeld radiation condition, which ensures that the scattered field is outgoing
(see, e.g., \cite{CK13}). Further, (\ref{eq2}) is the boundary condition imposed on $\pa D$, which depends
on the physical property of the obstacle $D$:
\be\label{eq22}
\left\{\begin{array}{ll}
\mathscr{B}u^{s}=u^{s} & \mbox{for a sound-soft obstacle $D$},\\
\ds\mathscr{B}u^{s}=\pa u^{s}/\pa\nu+\rho u^{s} & \mbox{for an impedance obstacle $D$},
\end{array}\right.
\en
where $\nu$ is the outward unit normal vector on the boundary $\pa D$ and
$\rho\in L^\infty(\pa D)$ is the impedance function on the boundary $\pa D$ which is complex-valued
with $\I(\rho)\geq 0$ almost everywhere on $\pa D$. In particular, when $\rho=0$, the impedance boundary
condition is reduced to the Neumann boundary condition, which corresponds to the sound-hard obstacle.

If $D$ is filled with an inhomogeneous medium characterized by the refractive index $n$,
then the scattering problem is modeled by the medium scattering problem
\begin{align}\label{eq4}
\qquad\qquad\qquad\qquad
\Delta u^s+k^2n u^s=&\;g&&\quad\mbox{in}\;\;\R^2,
\qquad\qquad\qquad\\ \label{eq5}
\lim\limits_{r\rightarrow\infty}r^{\frac{1}{2}}
\left(\frac{\pa u^s}{\pa r}-iku^s\right)=&\;0,&&\quad r=|x|,
\end{align}
where $g:=-k^2(n-1)u^i$.
In this paper, we assume that the contrast function $m:=n-1$ is supported in $\ov{D}$
and $n\in L^\infty(\R^2)$ with $\Rt[n(x)]\geq c_0>0$ for a constant $c_0$ and $\I[n(x)]\geq 0$
for almost all $x\in D$.

By the variational method \cite{CC14} or the integral equation method \cite{CK83,CK13},
it can be shown that the obstacle scattering problem (\ref{eq1})-(\ref{eq3}) and
the medium scattering problem (\ref{eq4})-(\ref{eq5}) have a unique solution.
In particular, it is well-known that the scattered field $u^s$ has the asymptotic behavior
\be\label{eq36}
u^s(x)=\frac{e^{i\pi/4}}{\sqrt{8 k\pi}}\frac{e^{ik|x|}}{\sqrt{|x|}}\left\{u^\infty(\hat{x})+O\left(\frac{1}{|x|}\right)\right\},
\quad|x|\rightarrow\infty
\en
uniformly for all observation directions $\hat{x}:=x/|x|\in\Sp^1$ with $\Sp^1$ denoting the unit
circle in $\R^2$ (see \cite{KG08}). Here, $u^\infty(\hat{x})$ is called the far-field pattern of
the scattered field $u^s(x)$, which is an analytic function of $\hat{x}\in\Sp^1$ \cite{CK13}.
In this paper, we consider the incident plane wave $u^i=u^i(x,d):=e^{i k x\cdot d}$,
where $d\in\Sp^1$ is the incident direction. Accordingly, the total field, the scattered field and
the far-field pattern are denoted by $u(x,d)$, $u^s(x,d)$ and $u^\infty(\hat{x},d)$, respectively.

To give a precise description of the inverse problem considered in this paper, let $B_{R}$
be the circle of radius $R$ and centered at $(0,0)$. Throughout this paper, we assume that $R$ is
large enough so that $\ov{D}\subset B_{R}$.
Then we shall consider the following inverse scattering problem with the phaseless {\color{HW}total-field} data
at a fixed frequency.

{\bf Inverse Problem (IP)}: Reconstruct the unknown scatterer $D$ from the measured phaseless {\color{HW}total-field}
data $|u(x,d)|$ for all $x\in \pa B_{R}$ and $d\in\Sp^1$, where $u(x,d)$ is the total field of the
scattering problem by the impenetrable obstacle or the inhomogeneous medium,
associated with the incident plane wave $u^i(x,d)$ at a fixed frequency.

Our purpose is to develop an approximate factorization method for solving the inverse problem (IP)
with the radius $R$ large enough, based on the asymptotic property in the linear space
$\mathcal{L}(H^{1/2}(\Sp^1), H^{-1/2}(\Sp^1))$
of the phaseless {\color{HW}total-field} operator defined in terms of the phaseless {\color{HW}total-field} data measured on
the circle $\pa B_R$ with large enough $R$,
together with the factorization of a modified far-field operator.
To this end, we introduce the following notations which will be used in the rest of the paper.
For any $x,d\in\Sp^1$, let $\hat{x}=(\cos\theta_{\hat{x}},\sin\theta_{\hat{x}})$,
$d=(\cos\theta_{d},\sin\theta_{d})$ with $\theta_{\hat{x}},\theta_d\in[0,2\pi]$.
Denote by $(\cdot,\cdot)_{L^2(\Sp^1)}$ the inner product in the Hilbert space $L^2(\Sp^1)$
and by $\langle\cdot,\cdot\rangle$ the duality pairing between $H^{-1/2}(\Sp^1)$ and $H^{1/2}(\Sp^1)$
extending the inner product $(\cdot,\cdot)_{L^2(\Sp^1)}$ (see, e.g., \cite{KG08}).
Throughout this paper, the constants may be different at different places.
We remark that, to the best of our knowledge, no uniqueness result is available yet
for our inverse problem.

\section{The asymptotic property of the phaseless {\color{HW}total-field} operator}\label{sec2}
\setcounter{equation}{0}

In this section, we study the asymptotic property in the linear space
$\mathcal{L}(H^{1/2}(\Sp^1), H^{-1/2}(\Sp^1))$ of the phaseless {\color{HW}total-field} operator defined in terms of
the phaseless {\color{HW}total-field} data measured on the circle $\pa B_R$ with large enough $R$,
which plays an essential role in the development of the approximate factorization method
for the inverse problem (IP). Precisely, we introduce the phaseless {\color{HW}total-field} operator
$N^{PW}_{R}: H^r(\Sp^1)\rightarrow H^s(\Sp^1)$, $r,s\in\R$, by
\be\label{eq9}
\left(N^{PW}_{R}\varphi\right)(\hat{x}):=\int_{\Sp^1}\left(|u(R\hat{x},d)|^2-1\right)
e^{ikR\hat{x}\cdot d}\varphi(d)ds(d),\qquad\;\varphi\in H^r(\Sp^1),
\en
where $\hat{x}=x/R$, $x\in\pa B_{R}$, $d\in\Sp^1$, $u(x,d)=u^i(x,d)+u^s(x,d)$ and $u^s(x,d)$ are the total
and scattered fields of the scattering problem by the impenetrable obstacle or the inhomogeneous medium,
associated with incident plane waves $u^i(x,d)$.
We also introduce the far-field operator $F: H^r(\Sp^1)\rightarrow H^s(\Sp^1)$, $r,s\in\R$, by
\be\label{eq34}
\left(F\varphi\right)(\hat{x}):=\int_{\Sp^1}u^\infty(\hat{x},d)\varphi(d)ds(d),
\qquad\;\varphi\in H^r(\Sp^1),
\en
where $u^\infty(\hat{x},d)$ for $\hat{x},d\in\Sp^1$ is the far-field pattern of the scattered field $u^s(x,d)$.
It can be seen that $N^{PW}_{R}$ and $F$ are well defined since $|u(x,d)|$ is analytic in
$x\in\pa B_{R}$ and in $d\in\Sp^1$, respectively, and $u^\infty(\hat{x},d)$ is analytic in
$\hat{x}\in\Sp^1$ and in $d\in\Sp^1$, respectively (see, e.g., \cite{CK13}).
In what follows, we will study the relationship between $N^{PW}_{R}$ and $F$ when $R$ is sufficiently large.

We need the following result on the property of the scattered field $u^s$.

\begin{lemma}\label{lem1}
For any $x\in \R^2\ba\ov{D}$ with $|x|$ large enough and $d\in\Sp^1$,
the scattered field $u^s(x,d)$ has the asymptotic behavior
\be\label{eq6}
u^s(x,d)=\frac{e^{i\pi/4}}{\sqrt{8 k\pi}}\frac{e^{ik|x|}}{|x|^{1/2}}u^\infty(\hat{x},d)+u^s_{Res}(x,d)
\en
with
\be\label{eq7}
\|u^\infty(\cdot,d)\|_{C^1({\Sp^1})}&\le& C,\\
\label{eq8}
|u^s_{Res}(x,d)|&\le& \frac{C}{|x|^{3/2}},
\en
where $C>0$ is a constant independent of $x$ and $d$.
\end{lemma}

\begin{proof}
This lemma can be easily obtained by using the well-posedness of the obstacle scattering problem
(\ref{eq1})-(\ref{eq3}) and the medium scattering problem (\ref{eq4})-(\ref{eq5}), together with
the asymptotic behavior (\ref{eq36}) of the scattered field $u^s$ (see, e.g., \cite{CK13}).
\end{proof}

{\color{HW}From (\ref{eq6}) it follows that
\ben
\left(|u(x,d)|^2-1\right)e^{ikx\cdot d}&=&\left(|e^{ikx\cdot d}+u^s(x,d)|^2-1\right)e^{ikx\cdot d}\\
&=& u^s(x,d)+\ov{u^s(x,d)}e^{2i k x\cdot d}+|u^s(x,d)|^2 e^{ikx\cdot d}\\
&=& \frac{e^{i\pi/4}}{\sqrt{8k\pi}}\frac{e^{ik|x|}}{|x|^{1/2}}u^\infty(\hat{x},d)
+\frac{e^{-i\pi/4}}{\sqrt{8k\pi}}\frac{e^{-ik|x|}}{|x|^{1/2}}\ov{u^\infty(\hat{x},d)}e^{2ikx\cdot d}\\
&&+ u^s_{Res}(x,d)+\ov{u^s_{Res}(x,d)}e^{2ikx\cdot d}+|u^s(x,d)|^2 e^{ikx\cdot d}.
\enn
Inserting this into (\ref{eq9}) gives}
\be\label{eq14}
\left(N^{PW}_{R}\varphi\right)(\hat{x})&=&\frac{e^{i\pi/4}}{\sqrt{8k\pi}}\frac{e^{ikR}}{R^{1/2}}
\left(F\varphi\right)(\hat{x})+\left(H^{(1)}_{PW,R}\varphi\right)(\hat{x})
+\left(H^{(2)}_{PW,R}\varphi\right)(\hat{x}),
\en
where
%$H^{(1)}_{PW,R}$ is given by
\be\label{eq10}
\left(H^{(1)}_{PW,R}\varphi\right)(\hat{x}):=\frac{e^{-i\pi/4}}{\sqrt{8k\pi}}\frac{e^{-ikR}}{R^{1/2}}
\left(L_{PW,R}\varphi\right)(\hat{x})
\en
with
\be\label{eq15}
\left(L_{PW,R}\varphi\right)(\hat{x}):=
\int_{\Sp^1}\ov{u^\infty(\hat{x},d)}e^{2ikR\hat{x}\cdot d}\varphi(d)ds(d)
\en
and
\be\label{eq16}
&&\left(H^{(2)}_{PW,R}\varphi\right)(\hat{x})\nonumber\\
&&:=\int_{\Sp^1}\left[u^s_{Res}(R\hat{x},d)+\ov{u^s_{Res}(R\hat{x},d)}e^{2ikR\hat{x}\cdot d}
+|u^s(R\hat{x},d)|^2 e^{ik R\hat{x}\cdot d}\right]\varphi(d)ds(d).\;\;
\en

To proceed further we need the following result for oscillatory integrals (see \cite{CH2017}).

\begin{lemma}[Lemma 3.9 in \cite{CH2017}]\label{lem2}
For any $-\infty<a<b<\infty$ let $u\in C^2[a,b]$ be real-valued and satisfy that
$|u'(t)|\geq1$ for all $t\in(a,b)$. Assume that $a=x_0<x_1<\cdots<x_N=b$ is a division of $(a,b)$
such that $u'$ is monotone in each interval $(x_{i-1},x_i)$, $i=1,\ldots,N$.
Then for any function $\phi$ defined on $(a,b)$ with integrable derivative and for any $\la>0$,
\ben
\left|\int_a^b e^{i\la u(t)}\phi(t)dt\right|\leq(2N+2)\la^{-1}\left[|\phi(b)|+\int^b_a|\phi'(t)|dt\right].
\enn
\end{lemma}

We can now study the property of $H^{(i)}_{PW,R}$, $i=1,2$, for $R$ large enough.

\begin{lemma}\label{lem3}
For any $\varphi\in H^1(\Sp^1)$ and for $R>0$ large enough we have
\ben
\|H^{(1)}_{PW,R}\varphi\|_{L^2(\Sp^1)}\le\frac{C}{R}\|\varphi\|_{H^1(\Sp^1)},
\enn
where $C>0$ is a constant independent of $R$.
\end{lemma}

\begin{proof}
From (\ref{eq10}) and the fact that $C^\infty(\Sp^1)$ is dense in $H^1(\Sp^1)$, it suffices to show
that for any $\varphi\in C^\infty(\Sp^1)$ and $R$ large enough,
\be\label{eq13}
\|L_{PW,R}\varphi\|_{L^2(\Sp^1)}\leq \frac{C}{R^{1/2}}\|\varphi\|_{H^1(\Sp^1)}.
\en

Let $\varphi$ be arbitrarily fixed function in $C^\infty(\Sp^1)$.
For $\hat{x},d\in\Sp^1$, let $\theta_{\hat{x}},\theta_d$ be the real numbers as defined at the end
of last section. Then, by the change of variables we have
\ben
\left(L_{PW,R}\varphi\right)(\hat{x})=\int^{2\pi}_{0}e^{2ikR\cos(\theta_{\hat{x}}-\theta_d)}
\ov{\wid{u}^\infty(\theta_{\hat{x}},\theta_d)}\wid{\varphi}(\theta_d)d\theta_d
=:\left(\wid{L}_{PW,R}\wid{\varphi}\right)(\theta_{\hat{x}}),
\enn
where
%$\wid{u}^\infty$ and $\wid{\varphi}$ are denoted by
$\wid{u}^\infty(\theta_{\hat{x}},\theta_d):=u^\infty(\hat{x},d)$
and $\wid{\varphi}(\theta_d):=\varphi(d)$ for $\theta_{\hat{x}},\theta_d\in [0,2\pi]$.
Let us define
$\wid{l}(\theta_{\hat{x}},\theta_d):=\ov{\wid{u}^\infty(\theta_{\hat{x}},\theta_d)}\wid{\varphi}(\theta_d)$
for $\theta_{\hat{x}},\theta_d\in [0,2\pi]$.
Since $u^\infty(\hat{x},d)$ is analytic in $d\in\Sp^1$ and $\varphi\in C^\infty(\Sp^1)$,
then $\wid{l}(\theta_{\hat{x}},\theta_d)$ and $\wid{\varphi}(\theta_d)$ can be extended as $C^\infty$-smooth
functions on $\R$ and $2\pi$-periodic with respect to $\theta_d$.
Then it follows by the change of variables that for $\theta_{\hat{x}}\in[0,2\pi]$,
\ben
&&\left(\wid{L}_{PW,R}\wid{\varphi}\right)(\theta_{\hat{x}})\\
&&\qquad\;=\left[\int^{\theta_{\hat{x}}-\frac{\pi}{2}}_{\theta_{\hat{x}}-\pi}
+\int^{\theta_{\hat{x}}}_{\theta_{\hat{x}}-\frac{\pi}{2}}
+\int^{\theta_{\hat{x}}+\frac{\pi}{2}}_{\theta_{\hat{x}}}
+\int^{\theta_{\hat{x}}+\pi}_{\theta_{\hat{x}}+\frac{\pi}{2}}\right]e^{2ikR\cos(\theta_{\hat{x}}-\theta_d)}
\ov{\wid{u}^\infty(\theta_{\hat{x}},\theta_d)}\wid{\varphi}(\theta_d)d\theta_d\\
&&\qquad\;=\int^{\frac{\pi}{2}}_0 e^{-2ikR\cos(t)}\wid{l}(\theta_{\hat{x}},t+\theta_{\hat{x}}-\pi)dt
+\int^{\frac{\pi}{2}}_0 e^{2ikR\cos (t)}\wid{l}(\theta_{\hat{x}},\theta_{\hat{x}}-t)dt\\
&&\qquad\;\;\;+\int^{\frac{\pi}{2}}_0 e^{2ikR\cos (t)}\wid{l}(\theta_{\hat{x}},t+\theta_{\hat{x}})dt
+\int^{\frac{\pi}{2}}_0 e^{-2ikR\cos (t)}\wid{l}(\theta_{\hat{x}},\theta_{\hat{x}}+\pi-t)dt\\
&&\qquad\;=:I_{1,PW}(\theta_{\hat{x}})+I_{2,PW}(\theta_{\hat{x}})+I_{3,PW}(\theta_{\hat{x}})
+I_{4,PW}(\theta_{\hat{x}}).
\enn

We now estimate $I_{i,PW}$, $i=1,2,3,4$. We first consider $I_{1,PW}$.
Let $\delta>0$ be small enough so that $\sin\delta\ge\delta/2$ and let $R$ be large enough.
Let $\theta_{\hat{x}}\in[0,2\pi]$ be arbitrarily fixed.
Then we have
\ben
I_{1,PW}(\theta_{\hat{x}})&=&\left[\int^{\delta}_0+\int^{\frac{\pi}{2}}_{\delta}\right]
e^{-2ikR\cos(t)}\wid{l}(\theta_{\hat{x}},t+\theta_{\hat{x}}-\pi)dt
=:I^{(1)}_{1,PW}(\theta_{\hat{x}})+I^{(2)}_{1,PW}(\theta_{\hat{x}}).
\enn
Set $f(t)=-2\cos(t)/\delta$. Then $f'(t)=2\sin (t)/\delta$. Thus, for $t\in[\delta,\pi/2]$
we have $|f'(t)|=2\sin(t)/\delta\geq 2\sin\delta/\delta\geq 1$ and $f'(t)$ is monotone in $[\delta,\pi/2]$.
By using Lemma \ref{lem2}, the formula (\ref{eq7}) and the reciprocity relation
that $u^\infty(\hat{x},d)=u^\infty(-d,-\hat{x})$ for $\hat{x},d\in\Sp^1$ (see, e.g., \cite{KG08}),
it is obtained that
\ben
|I^{(2)}_{1,PW}(\theta_{\hat{x}})|&=&\left|\int^{\frac{\pi}{2}}_\delta e^{ikR\delta f(t)}
\wid{l}(\theta_{\hat{x}},t+\theta_{\hat{x}}-\pi)dt\right|\\
&\le&\frac{C}{R\delta}\left(|\wid{l}(\theta_{\hat{x}},\theta_{\hat{x}}-\frac{\pi}{2})|
+\int^{\frac{\pi}{2}}_\delta\left|\frac{d}{dt}\wid{l}(\theta_{\hat{x}},t+\theta_{{\hat{x}}}-\pi)\right|dt\right)\\
&\le&\frac{C}{R\delta}\left(|\wid{\varphi}(\theta_{\hat{x}}-\frac{\pi}{2})|
+\|\wid{\varphi}\|_{H^1[0,2\pi]}\right).
\enn
This yields that
\be\label{eq11}
\|I^{(2)}_{1,PW}\|_{L^2[0,2\pi]}\leq \frac{C}{R\delta}\|\widetilde{\varphi}\|_{H^1[0,2\pi]}.
\en
Further, it follows from the formula (\ref{eq7}) that
\ben
|I^{(1)}_{1,PW}(\theta_{\hat{x}})|&\leq&
\int^{\delta}_0\left|\wid{l}(\theta_{\hat{x}},t+\theta_{\hat{x}}-\pi)\right|dt\\
&\leq& \left(\int^\delta_0 1^2 dt\right)^{1/2}
\left(\int^\delta_0\left|\wid{l}(\theta_{\hat{x}},t+\theta_{\hat{x}}-\pi)\right|^2dt\right)^{1/2}\\
&\leq& C\delta^{\frac{1}{2}}\left(\int^\delta_0
\left|\wid{\varphi}(t+\theta_{\hat{x}}-\pi)\right|^2 dt\right)^{1/2},
\enn
which implies that
\be\label{eq12}
\|I^{(1)}_{1,PW}\|_{L^2[0,2\pi]}\le C\delta\|\wid{\varphi}\|_{L^2[0,2\pi]}.
\en
Using (\ref{eq11}) and (\ref{eq12}) and taking $\delta=R^{-{1}/{2}}$ give
\ben
\|I_{1,PW}\|_{L^2[0,2\pi]}\le C\delta\|\widetilde{\varphi}\|_{L^2[0,2\pi]}
+\frac{C}{R\delta}\|\widetilde{\varphi}\|_{H^1[0,2\pi]}
\leq \frac{C}{R^{1/2}}\|\widetilde{\varphi}\|_{H^1[0,2\pi]}.
\enn
By a similar argument we can obtain that
\ben
\|I_{i,PW}\|_{L^2[0,2\pi]}\le\frac{C}{R^{1/2}}\|\wid{\varphi}\|_{H^1[0,2\pi]},\quad\; i=2,3,4.
\enn
Hence it follows that
\ben
\|L_{PW,R}\varphi\|_{L^2(\Sp^1)}=\|\widetilde{L}_{PW,R}\widetilde{\varphi}\|_{L^2[0,2\pi]}
\leq \frac{C}{R^{1/2}}\|\widetilde{\varphi}\|_{H^1[0,2\pi]}
\leq\frac{C}{R^{1/2}}\|\varphi\|_{H^1{(\Sp^1)}},
\enn
that is, (\ref{eq13}) holds. The proof is thus complete.
\end{proof}

\begin{lemma}\label{lem4}
For any $\varphi\in L^2(\Sp^1)$ and for $R>0$ large enough, we have
\ben
\|H^{(2)}_{PW,R}\varphi\|_{L^2(\Sp^1)}\leq \frac{C}{R}\|\varphi\|_{L^2(\Sp^1)},
\enn
where $C>0$ is a constant independent of $R$.
\end{lemma}

\begin{proof}
Lemma \ref{lem1} gives that for any $\hat{x},d\in\Sp^1$, $|u^s(R\hat{x},d)|\leq{C}{R^{-1/2}}$
when $R$ is large enough. This, together with (\ref{eq8}) and (\ref{eq16}), implies the required result.
The proof is then completed.
\end{proof}

By the formula (\ref{eq14}) and Lemmas \ref{lem3} and \ref{lem4}, we can obtain the following result
on the relationship between the operators $N^{PW}_{R}$ and $F$ when $R$ is sufficiently large.

\begin{lemma}\label{lem5}
For $R>0$ large enough, we have
\ben
\left\|N^{PW}_{R}-\frac{e^{i\pi/4}}{\sqrt{8k\pi}}\frac{e^{ikR}}{R^{1/2}}F\right\|_{H^1(\Sp^1)
\rightarrow L^2(\Sp^1)}\leq\frac{C}{R},
\enn
where $C>0$ is a constant independent of $R$.
\end{lemma}

Denote by $\left(N^{PW}_{R}\right)^*$ and $F^*$ the adjoint operator of the operators $N^{PW}_{R}$ and $F$,
respectively. Then, by a similar argument as in the proof of Lemma \ref{lem5} we have the following result
on the relationship between the operators $\left(N^{PW}_{R}\right)^*$ and $F^*$ when $R$ is large enough.

\begin{lemma}\label{lem6}
For $R>0$ large enough, we have
\ben
\left\|\left(N^{PW}_{R}\right)^*-\frac{e^{-i\pi/4}}{\sqrt{8k\pi}}\frac{e^{-ikR}}{R^{1/2}}F^*\right\|_{H^1(\Sp^1)
\rightarrow L^2(\Sp^1)}\leq\frac{C}{R},
\enn
where $C>0$ is a constant independent of $R$.
\end{lemma}

\begin{proof}
By (\ref{eq14})-(\ref{eq16}) it follows that
\be\label{eq17}
\left(N^{PW}_{R}\right)^*&=&\frac{e^{-i\pi/4}}{\sqrt{8k\pi}}\frac{e^{-ikR}}{R^{1/2}}F^*
+H^{(1)*}_{PW,R}+H^{(2)*}_{PW,R},
\en
where $H^{(1)*}_{PW,R}$ and $H^{(2)*}_{PW,R}$ denote the adjoint operator of $H^{(1)}_{PW,R}$ and
$H^{(2)}_{PW,R}$, respectively, and are represented as follows:
for $d\in\Sp^1$ and $\psi\in H^r(\Sp^1)$ with $r\in\R$,
\ben
\left(H^{(1)*}_{PW,R}\psi\right)(d):=\frac{e^{i\pi/4}}{\sqrt{8k\pi}}\frac{e^{ikR}}{R^{1/2}}
\left(L^*_{PW,R}\psi\right)(d)
\enn
with
\ben
\left(L^*_{PW,R}\psi\right)(d)
:=\int_{\Sp^1}{u^\infty(\hat{x},d)}e^{-2ikR\hat{x}\cdot d}\psi(\hat{x})ds(\hat{x})
\enn
and
\ben
&&\left(H^{(2)*}_{PW,R}\psi\right)(d)\\
&&\quad:=\int_{\Sp^1}\left[\ov{u^s_{Res}(R\hat{x},d)}+{u^s_{Res}(R\hat{x},d)}e^{-2ikR\hat{x}\cdot d}
+|u^s(R\hat{x},d)|^2 e^{-ik R\hat{x}\cdot d}\right]\psi(\hat{x})ds(\hat{x}).
\enn
Similarly as in the proof of Lemmas \ref{lem3} and \ref{lem4}, we can apply Lemmas \ref{lem1} and \ref{lem2}
to obtain that
\be\label{eq18}
\left\|H^{(1)*}_{PW,R}\right\|_{H^1(\Sp^1)\rightarrow L^2(\Sp^1)}\leq\frac{C}{R},\qquad
\left\|H^{(2)*}_{PW,R}\right\|_{L^2(\Sp^1)\rightarrow L^2(\Sp^1)}\leq\frac{C}{R}
\en
for $R$ large enough. The required estimate then follows from (\ref{eq17}) and (\ref{eq18}).
The proof is thus complete.
\end{proof}

Making use of Lemmas \ref{lem5} and \ref{lem6}, we can prove the following theorem.

\begin{theorem}\label{thm1}
For $R>0$ large enough, we have
\ben
\left\|N^{PW}_{R}-\frac{e^{i\pi/4}}{\sqrt{8k\pi}}\frac{e^{ikR}}{R^{1/2}}
F\right\|_{H^{1/2}(\Sp^1)\rightarrow H^{-1/2}(\Sp^1)}\leq\frac{C}{R},
\enn
where $C>0$ is a constant independent of $R$.
\end{theorem}

\begin{proof}
Let $R$ be large enough. First, it follows from Lemma \ref{lem6} that
\ben
\left\|N^{PW}_{R}-\frac{e^{i\pi/4}}{\sqrt{8k\pi}}\frac{e^{ikR}}{R^{1/2}}
F\right\|_{L^2(\Sp^1)\rightarrow H^{-1}(\Sp^1)}\leq\frac{C}{R}.
\enn
This, together with Lemma \ref{lem5} and the interpolation property of Sobolev spaces
(see \cite[Theorem 8.13]{K14}), gives the required estimate.
The proof is thus complete.
\end{proof}

\begin{remark}\label{re3.8} {\rm {\color{HW}
From the above discussions, it can be seen that the constant $C$ in Theorem \ref{thm1} depends on the wave number $k$.
In many applications for inverse scattering problems with phaseless data, the wave number $k$ is usually large,
and thus it is interesting to study the explicit dependence on $k$ of the constant $C$.
However, this is challenging since, to do so, we need to study the explicit dependence on $k$ of
the far-field pattern $u^\infty(\hat{x},d)$ (note that the constant $C$ also depends on $u^\infty(\hat{x},d)$).
As far as we know, there is no result available for the explicit dependence on $k$ of $u^\infty(\hat{x},d)$.
In fact, it is also very difficult to study the explicit dependence on $k$ of the solution of
the scattering problems considered.
So far, there are several results concerning the explicit dependence on $k$
of the solution of the scattering problems in some special cases
(see, e.g., \cite{CM08} for the case of a starlike sound-soft obstacle,
\cite{S14} for the case of an impedance obstacle with the special impedance function $\rho=i\eta$ with a positive
constant $\eta$, and \cite{MS19} for the case of a starlike penetrable obstacle with a constant refractive index).
But, to the best of our knowledge, no result is available for more general cases.
}}
\end{remark}

\section{The approximate factorization method}\label{sec3}
\setcounter{equation}{0}

In this section, we make use of the asymptotic behavior of the phaseless {\color{HW}total-field} operator to develop
an approximate factorization method for the inverse problem.

{\color{HW}
From Theorem \ref{thm1} we know that the asymptotic property of the phaseless total-field operator $N^{PW}_R$ holds
only for the case when $N^{PW}_R$ is an operator from $H^{1/2}(\Sp^1)$ to $H^{-1/2}(\Sp^1)$.
Then, by the classical factorization method \cite{KG08} we need to modify $N^{PW}_R$
by introducing the operators $B_{1/2}$ and $B^*_{1/2}$ defined below so that the modified operator $\wid{N}^{PW}_{R}$
(see (\ref{eq37}) below) is an operator from $L^2(\Sp^1)$ to $L^2(\Sp^1)$. Thus,
an approximate factorization method can be developed for the inverse problem by applying
the classical factorization method \cite{KG08} to the modified far-field operator $\wid{F}$ defined in (\ref{eq38})
(see Theorem \ref{thm3}) together with the asymptotic property of the modified operator $\wid{N}^{PW}_{R}$
in Theorem \ref{thm2} (see Remark \ref{rem2}).
}

We now introduce some auxiliary operators.
For $\hat{x}\in\Sp^1$ define $\varphi_m(\hat{x}):={1}/{\sqrt{2\pi}}e^{i m\theta_{\hat{x}}}$,
$m\in\mathbb{Z}$, where $\theta_{\hat{x}}\in[0,2\pi]$ is defined as above. It is well known that
$\{\varphi_m:m\in\mathbb{Z}\}$ is a complete orthonormal system in $L^2(\Sp^1)$.
Thus, for any $\varphi\in L^2(\Sp^1)$ we have that, in the sense of mean square convergence,
\be\label{eq19}
\varphi(\hat{x})=\sum\limits^{+\infty}_{m=-\infty} a_m\varphi_m(\hat{x}),\qquad a_m:=(\varphi,\varphi_m)_{L^2(\Sp^1)}=\int_{\Sp^1}\varphi(\hat{x})\ov{\varphi_m(\hat{x})}ds(\hat{x}).
\en
Further, it is known that $H^r(\Sp^1)$ with $r\geq0$ is a Hilbert space under the norm
$\|\varphi\|_{H^r(\Sp^1)}:=[\sum^{+\infty}_{m=-\infty}(1+m^2)^r|a_m|^2]^{1/2}$ for $\varphi\in H^r(\Sp^1)$
with the coefficients $a_m$ given in (\ref{eq19}).
For more properties of the Sobolev Space $H^r(\Sp^1)$, $r\geq0$, and its
dual space $H^{-r}(\Sp^1)$, the reader is referred to \cite{CC14,K14}.
Let $B_{1/2}$ be the operator defined by
\ben
B_{1/2}\varphi:=\sum\limits^{+\infty}_{m=-\infty}(1+m^2)^{-1/4}a_m\varphi_m
\enn
for $\varphi\in L^2(\Sp^1)$ with the coefficients $a_m$ given in (\ref{eq19}) and let $B^*_{1/2}$ be
the adjoint of $B_{1/2}$. Then we have the following results concerning $B_{1/2}$ and $B^*_{1/2}$.

\begin{lemma}\label{lem7}
$B_{1/2}$ is bijective (and so boundedly invertible) from $L^2(\Sp^1)$ to $H^{1/2}(\Sp^1)$.
Further, $B^*_{1/2}$ is bijective (and so boundedly invertible) from $H^{-1/2}(\Sp^1)$ to $L^2(\Sp^1)$
and given by
\be\label{eq20}
B^*_{1/2}\psi =\sum\limits^{+\infty}_{m=-\infty} (1+m^2)^{-1/4}b_m\varphi_m,\quad
b_m:=\langle\psi,\varphi_m\rangle=\int_{\Sp^1}\psi(\hat{x})\ov{\varphi_m(\hat{x})}ds(\hat{x})\;\;
\en
for $\psi\in H^{-1/2}(\Sp^1)$, where $\langle\cdot,\cdot\rangle$ is the duality pair between
$H^{-1/2}(\Sp^1)$ and $H^{1/2}(\Sp^1)$.
\end{lemma}

\begin{proof}
Let $\varphi\in L^2(\Sp^1)$ with the coefficients $a_m$ given in (\ref{eq19}). Then we have
\ben
\|B_{1/2}\varphi\|^2_{H^{1/2}(\Sp^1)}
&=&\sum^{+\infty}_{m=-\infty}(1+m^2)^{1/2}|(1+m^2)^{-1/4}a_m|^2\\
&=&\sum^{+\infty}_{m=-\infty}|a_m|^2=\|\varphi\|^2_{L^2(\Sp^1)}<+\infty.
\enn
This implies that $B_{1/2}$ is a bounded operator from $L^2(\Sp^1)$ to $H^{1/2}(\Sp^1)$.
For $\varphi\in L^2(\Sp^1)$ with the coefficients $a_m$ given in (\ref{eq19}) define $B_{-1/2}$ by
\ben
B_{-1/2}\varphi:=\sum^{+\infty}_{m=-\infty}(1+m^2)^{1/4}a_m\varphi_m.
\enn
Similarly as above, we can deduce that $B_{-1/2}$ is a bounded operator from $H^{1/2}(\Sp^1)$ to $L^2(\Sp^1)$.
It is easily seen that $B_{-1/2}B_{1/2}\varphi=\varphi$ for $\varphi\in L^2(\Sp^1)$ and
$B_{1/2}B_{-1/2}\varphi=\varphi$ for $\varphi\in H^{1/2}(\Sp^1)$. Then
$B_{1/2}$ is bijective (and so boundedly invertible) from $L^2(\Sp^1)$ to $H^{1/2}(\Sp^1)$,
and thus $B^*_{1/2}$ is also bijective (and so boundedly invertible) from  $H^{-1/2}(\Sp^1)$ to $L^2(\Sp^1)$.
Further, let $\psi\in H^{-1/2}(\Sp^1)$ and let $b_m$ be given in (\ref{eq20}). Then we have
\ben
\langle B_{1/2}\varphi,\psi\rangle
&=&\left\langle\sum\limits^{+\infty}_{m=-\infty}(1+m^2)^{-1/4}a_m\varphi_m,\psi\right\rangle\\
&=&\sum\limits^{+\infty}_{m=-\infty}(1+m^2)^{-1/4}(\varphi,\varphi_m)_{L^2(\Sp^1)}\langle\varphi_m,\psi\rangle\\
&=&\sum^{+\infty}_{m=-\infty}(1+m^2)^{-1/4}(\varphi,\varphi_m)_{L^2(\Sp^1)}\ov{b}_m\\
&=&\left(\varphi,\sum\limits^{+\infty}_{m=-\infty}(1+m^2)^{-1/4}b_m\varphi_m\right)_{L^2(\Sp^1)}.
\enn
Therefore, $B^*_{1/2}$ has the form (\ref{eq20}). This completes the proof.
\end{proof}

With these preparations, we introduce the modified phaseless {\color{HW}total-field} operator
$\wid{N}^{PW}_{R}$ and the modified far-field operator $\wid{F}$ by
\be\label{eq37}
\wid{N}^{PW}_{R}&:=&e^{-i(kR+\pi/4)}B^*_{1/2}N^{PW}_{R} B_{1/2},\\ \label{eq38}
\wid{F}&:=&B^*_{1/2}F B_{1/2},
\en
respectively. The approximate factorization method for the inverse problem will be developed with utilizing
the asymptotic property of the operator $\wid{N}^{PW}_{R}$ for $R$ large enough, as discussed at the beginning
of this section and seen in Remark \ref{rem2} below.

From the property of $B_{1/2}$ and $B^*_{1/2}$, we know that $\wid{N}^{PW}_{R}$ and $\wid{F}$ are bounded
operators from $L^2(\Sp^1)$ to $L^2(\Sp^1)$. Further, with the aid of Theorem \ref{thm1} and Lemma \ref{lem7},
we can obtain the following theorem on the asymptotic property of $\wid{N}^{PW}_{R}$ and $\wid{F}$
for $R$ large enough.

\begin{theorem}\label{thm2}
For $R>0$ large enough we have
\ben
\left\|\wid{N}^{PW}_{R}-\frac{1}{\sqrt{8k\pi R}}\wid{F}\right\|_{L^2(\Sp^1)\rightarrow L^2(\Sp^1)}
\le\frac{C}{R},
\enn
where $C>0$ is a constant independent of $R$.
\end{theorem}

\begin{remark}\label{rem1} {\rm
If the far-field operator $F$ is regarded as an operator from $L^2(\Sp^1)$ to $L^2(\Sp^1)$,
then the modified far-field operator $\wid{F}$ can be rewritten as
\be\label{eq25}
\wid{F}=B^*_{1/2}I^*_0 F I_0 B_{1/2}.
\en
Here, $I_0$ is the imbedding operator from $H^{1/2}(\Sp^1)$ to $L^2(\Sp^1)$ and its adjoint $I^*_0$ is
an imbedding operator from $L^2(\Sp^1)$ to $H^{-1/2}(\Sp^1)$.
From \cite[Chapter 8]{K14} it is seen that $I_0$ is injective and compact with a dense range in $L^2(\Sp^1)$
and $I^*_0$ is injective and compact with a dense range in $H^{-1/2}(\Sp^1)$.
Note that the formula (\ref{eq25}) will be used in the study of the operator $\wid{F}$
(see Theorem \ref{thm3} for details).
}
\end{remark}

From Theorem \ref{thm2} it is known that the leading order term of the operator $\wid{N}^{PW}_{R}$ is
${(1/\sqrt{8k\pi R})}\wid{F}$ as $R\rightarrow+\infty$ in the linear space
$\mathcal{L}(L^2(\Sp^1), L^2(\Sp^1))$ of bounded linear operators from $L^2(\Sp^1)$ to $L^2(\Sp^1)$.
Note that the coefficient ${(1/\sqrt{8k\pi R})}$ is a constant for arbitrarily fixed $R$.
Thus, instead of studying the operator $\wid{N}^{PW}_{R}$ directly, we will investigate the property of
the operator $\wid{F}$, making use of the factorization method presented in \cite{KG08,KL14},
where the factorization of the far-field operator $F$ has been extensively investigated for
inverse obstacle scattering problems.
In what follows, we will employ some useful results in \cite{KG08,KL14} to derive a characterization
of the obstacle $D$ from the operator $\wid{F}$.

Define the boundary integral operators $S,K,K':H^{-1/2}(\pa D)\rightarrow H^{1/2}(\pa D)$ and
$T:H^{1/2}(\pa D)\rightarrow H^{-1/2}(\pa D)$ by
\be\label{eq21}
(S\varphi)(x)&:=&\int_{\pa D}\Phi(x,y)\varphi(y)ds(y),\quad x\in\pa D,\\ \label{eq21a}
(K\varphi)(x)&:=&\int_{\pa D}\frac{\pa\Phi(x,y)}{\pa\nu(y)}\varphi(y)ds(y),\quad x\in\pa D,\\ \label{eq21b}
(K'\varphi)(x)&:=&\int_{\pa D}\frac{\pa\Phi(x,y)}{\pa\nu(x)}\varphi(y)ds(y),\quad x\in\pa D,\\ \label{eq24}
(T\psi)(x)&:=&\frac{\pa}{\pa\nu(x)}\int_{\pa D}\frac{\pa\Phi(x,y)}{\pa\nu(y)}\psi(y)ds(y),\quad x\in\pa D,
\en
where $\Phi(x,y):=(i/4)H^{(1)}_0(k|x-y|)$, $x,y\in\R^2$, $x\neq y$, is the fundamental solution of
the Helmholtz equation $\Delta w+k^2 w=0$ in $\R^2$.
Here, $H^{(1)}_0$ is the Hankel function of the first kind of order $0$.
From \cite{CK13,KG08,M00} it is known that the boundary integral operators
$S,K,K':H^{-1/2}(\pa D)\rightarrow H^{1/2}(\pa D)$ and $T:H^{1/2}(\pa D)\rightarrow H^{-1/2}(\pa D)$
are bounded operators.

We now collect some results in \cite{KG08} for the factorization of the far-field operator $F$
in the cases of a sound-soft obstacle, an impedance obstacle and an inhomogeneous medium.

\begin{lemma}\label{lem8}
Let $D$ be a sound-soft obstacle. Assume that $k^2$ is not a Dirichlet eigenvalue of $-\Delta$ in D.
If the far-field operator $F$ is regarded as the operator from $L^2(\Sp^1)$ to $L^2(\Sp^1)$, then
the following statements hold.
\begin{enumerate}[(a)]
 \item
  The operator $F$ has the factorization
  \ben
  F=-G_{Dir} S^* G^*_{Dir},
  \enn
  where $G_{Dir}:H^{1/2}(\pa D)\rightarrow L^2(\Sp^1)$ is the data-to-pattern operator given by
  $G_{Dir}f=w^\infty$ with $w^\infty\in L^2(\Sp^1)$ being the far-field pattern of the scattered field
  $w^s$ of the exterior Dirichlet problem $(\ref{eq1})-(\ref{eq3})$ with boundary value $f\in H^{1/2}(\pa D)$.
  Here, $G^*_{Dir}: L^2(\Sp^1)\rightarrow H^{-1/2}(\pa D)$ and $S^*: H^{-1/2}(\pa D)\rightarrow H^{1/2}(\pa D)$
  are the adjoint of $G_{Dir}$ and $S$, respectively.
  \item
  The operator $G_{Dir}$ is compact, one-to-one with dense range in $L^2(\Sp^1)$. For any $z\in\R^2$
  define the function $\phi_z\in L^2(\Sp^1)$ by
  \be\label{eq23}
  \phi_z(\hat{x}):=e^{-ik\hat{x}\cdot z},\quad\hat{x}\in\Sp^1.
  \en
  Then $\phi_z$ belongs to the range $\mathcal{R}(G_{Dir})$ of $G_{Dir}$ if and only if $z\in D$.
  \item
  The operator $S$ has the following properties.
  \begin{enumerate}[(i)]
  \item
  The operator $S$ is an isomorphism from $H^{-1/2}(\pa D)$ into $H^{1/2}(\pa D)$.
  \item
  Let $S_i$ be defined by $(\ref{eq21})$ with $k=i$.
  Then $S_i$ is self-adjoint and coercive as an operator from $H^{-1/2}(\pa D)$ into $H^{1/2}(\pa D)$,
  that is, there exists $c_1>0$ with $\langle\varphi,S_i\varphi\rangle\geq c_1\|\varphi\|^2_{H^{-1/2}(\pa D)}$
  for all $\varphi\in H^{-1/2}(\pa D)$.
  \item
  $\I\langle\varphi,S\varphi\rangle < 0$ for all $\varphi\in H^{-1/2}(\pa D)$ with $\varphi\neq0$.
  \item
  The difference $S-S_i$ is compact from $H^{-1/2}(\pa D)$ into $H^{1/2}(\pa D)$.
  \end{enumerate}
\end{enumerate}
\end{lemma}

\begin{proof}
(a) was proved in \cite[Theorem 1.15]{KG08}, (b) was proved in \cite{KG08} as Theorem 1.12
and Lemma 1.13, and (c) was proved in \cite[Lemma 1.14]{KG08}.
\end{proof}

\begin{lemma}\label{lem9}
Let $D$ be an obstacle with the impedance boundary condition given in $(\ref{eq22})$ with
$\rho\in L^\infty(\pa D)$ and $\I(\rho)\geq 0$ almost everywhere on $\pa D$. Assume that $k^2$
is not an eigenvalue of $-\Delta$ in $D$ with respect to the impedance boundary condition.
If the far-field operator $F$ is regarded as the operator from $L^2(\Sp^1)$ to $L^2(\Sp^1)$,
then the following statements hold.
\begin{enumerate}[(a)]
  \item
  The operator $F$ has the factorization
  \ben
  F=-G_{imp}T^*_{imp}G^*_{imp},
  \enn
  where $G_{imp}:H^{-1/2}(\pa D)\rightarrow L^2(\Sp^1)$ is the data-to-pattern operator given by
  $G_{imp}f=w^\infty$ with $w^\infty\in L^2(\Sp^1)$ being the far-field pattern of the scattered field
  $w^s$ of the exterior impedance problem $(\ref{eq1})-(\ref{eq3})$ with boundary value $f\in H^{-1/2}(\pa D)$
  and $T_{imp}:H^{1/2}(\pa D)\rightarrow H^{-1/2}(\pa D)$ is given by
  $T_{imp}=T+i(\I \rho)I+K'\ov{\rho}+\rho K+\rho S\ov{\rho}$.
  Here, $G^*_{imp}:L^2(\Sp^1)\rightarrow H^{1/2}(\pa D)$ and
  $T^*_{imp}: H^{1/2}(\pa D)\rightarrow H^{-1/2}(\pa D)$ are the adjoint of $G_{imp}$ and $T_{imp}$,
  respectively.
  \item
  The operator $G_{imp}$ is compact, one-to-one with dense range in $L^2(\Sp^1)$.
  For any $z\in\R^2$, $\phi_z$ belongs to the range $\mathcal{R}(G_{imp})$ of $G_{imp}$ if and only
  if $z\in D$, where $\phi_z$ is the function defined in $(\ref{eq23})$.
  \item
  The operator $T_{imp}$ has the following properties.
  \begin{enumerate}[(i)]
  \item
  The operator $T_{imp}$ is an isomorphism from $H^{1/2}(\pa D)$ into $H^{-1/2}(\pa D)$.
  \item
  Let $T_i$ be defined by (\ref{eq24}) with $k=i$. Then $-T_i$ is self-adjoint and coercive as an operator
  from $H^{1/2}(\pa D)$ into $H^{-1/2}(\pa D)$, that is, there exists $c_1>0$ with
  $-\langle T_i\varphi,\varphi\rangle\geq c_1\|\varphi\|^2_{H^{1/2}(\pa D)}$ for all
  $\varphi\in H^{1/2}(\pa D)$.
  \item
  $\I\langle T_{imp}\varphi,\varphi\rangle>0$ for all $\varphi\in H^{1/2}(\pa D)$ with $\varphi\neq0$.
  \item
  The difference $T_{imp}-T_i$ is compact from $H^{1/2}(\pa D)$ into $H^{-1/2}(\pa D)$.
  \end{enumerate}
\end{enumerate}
\end{lemma}

\begin{proof}
(a) was proved in \cite[Theorem 2.6]{KG08}, (b) can be shown by using \cite[Theorem 2.4]{KG08} and
the argument as in the proof of \cite[Theorem 2.8]{KG08}, and (c) can be proved by using
\cite[Theorem 1.26]{KG08}, \cite[Theorem 2.6]{KG08} and the argument as in the proof
of \cite[Lemma 2.7]{KG08}.
\end{proof}

\begin{lemma}\label{lem10}
Let $D$ be an inhomogeneous medium, where the refractive index $n\in L^\infty(D)$ satisfies that
$\Rt[n(x)]\ge c_0>0$ and $\I[n(x)]\ge 0$ for almost all $x\in D$ with some constant $c_0$ and
the contract function $m=n-1$ is compactly supported in $\ov{D}$ and satisfies that
$\Rt[m(x)]\ge c>0$ or $\Rt[m(x)]\leq -c<0$ for almost all $x\in D$ with some constant $c$.
Assume that $k^2$ is not an eigenvalue of the interior transmission problem in $D$ in the sense
of \cite[Definition 4.7]{KG08}. If the far-field operator $F$ is regarded as the operator from
$L^2(\Sp^1)$ to $L^2(\Sp^1)$, then the following statements hold.
\begin{enumerate}[(a)]
  \item
  The operator $F$ has the factorization
  \ben
  F=H^*_{med}T_{med}H_{med},
  \enn
  where $H_{med}:L^2(\Sp^1)\rightarrow L^2(D)$ and $H^*_{med}:L^2(D)\rightarrow L^2(\Sp^1)$ are defined as
  \ben
  &&(H_{med}\psi)(x)=\sqrt{|m(x)|}\int_{\Sp^1}e^{ik x\cdot d}\psi(d)ds(d),\quad x\in D,\\
  &&(H^*_{med}\varphi)(\hat{x})=\iint_{D}e^{-ik\hat{x}\cdot y}\sqrt{|m(y)|}\varphi(y)dy,\quad\hat{x}\in\Sp^1.
  \enn
  The operator $T_{med}: L^2(D)\rightarrow L^2(D)$ is defined by
  $T_{med}f=k^2{\rm sign}(m)[f+\sqrt{|m|}w|_D]$,
  where ${\rm sign}(m):=m/|m|$ and $w\in H^1_{loc}(\R^2)$ is the radiating solution of the equation
  \ben
  \Delta w+k^2(1+m)w&=&-k^2\frac{m}{\sqrt{|m|}}f\quad\mbox{in}\;\;\R^2.
  \enn
  \item
  The operator $H^*_{med}$ is compact with dense range in $L^2(\Sp^1)$.
  For any $z\in\R^2$, $\phi_z$ belongs to the range $\mathcal{R}(H^*_{med})$ of $H^*_{med}$
  if and only if $z\in D$, where $\phi_z$ is the function defined in (\ref{eq23}).
  \item
  The operator $T_{med}$ has the following properties.
  \begin{enumerate}[(i)]
    \item
    The operator $T_{med}$ can be written in the form $T_{med}=T^{(0)}_{med}+K_{med}$,
    where $T^{(0)}_{med}$ has the form $T^{(0)}_{med}=k^2({\rm sign}\;m)f$ for $f\in L^2(D)$
    and $K_{med}:L^2(D)\rightarrow L^2(D)$ is compact. For the case when $\Rt[m(x)]\geq c>0$ for almost all
    $x\in D$, $\Rt [T^{(0)}_{med}]$ is self-adjoint and coercive in $L^2(D)$.
    For the case when $\Rt[m(x)]\leq-c<0$ for almost all $x\in D$, $\Rt[-T^{(0)}_{med}]$ is self-adjoint
    and coercive in $L^2(D)$.
    \item
    We have $\I(T_{med}f,f)_{L^2(D)}\ge 0$ for all $f\in L^2(D)$.
    \item
    $\I(T_{med} f,f)_{L^2(D)}>0$ for all $f\in \ov{\mathcal{R}(H_{med})}$ with $f\neq0$.
  \end{enumerate}
\end{enumerate}
\end{lemma}

\begin{proof}
(a) was proved in \cite[Theorem 4.5]{KG08}, (b) can be proved by using \cite[Theorem 4.6]{KG08}
in conjunction with the compactness and injectivity of $H_{med}$, and (c) follows from \cite[Theorem 4.8]{KG08}.
\end{proof}

Using formula (\ref{eq25}) in conjunction with Lemmas \ref{lem8}, \ref{lem9} and \ref{lem10} and
the Range Identity \cite[Theorem 1.1]{KL14}, we can obtain the following theorem on the
characterization of the obstacle $D$, based on the factorization of the operator
$\wid{F}_{\#}:=|\Rt(\wid{F})|+|\I(\wid{F})|$.

\begin{theorem}\label{thm3}
(a) Let $D$ be a sound-soft obstacle and let us assume that the conditions in Lemma $\ref{lem8}$
are satisfied. Then we have
\be\label{eq26}
z \in D & \Longleftrightarrow& B^*_{1/2}\phi_{z}\in\mathcal{R}(\wid{F}_{\#}^{{1}/{2}}) \\ \label{eq27}
& \Longleftrightarrow& W(z):=\left[\sum^\infty_{j=1}{\left|\left(B^*_{1/2}\phi_{z},\psi_{j}\right)_{L^{2}
(\Sp^{1})}\right|^{2}}\Big/{\lambda_{j}}\right]^{-1}>0,
\en
where $\phi_z$ is the function defined in $(\ref{eq23})$ and $\{\lambda_j;\psi_j\}_{j\in\mathbb{N}}$ is
an eigensystem of the self-adjoint operator $\wid{F}_{\#}$.

(b) Let $D$ be an impedance obstacle and let us assume that the conditions in Lemma $\ref{lem9}$
are satisfied. Then the statements $(\ref{eq26})$ and $(\ref{eq27})$ hold.

(c) Let $D$ be filled with an inhomogeneous medium and let us assume that the conditions in
Lemma $\ref{lem10}$ are fulfilled. Then the statements $(\ref{eq26})$ and $(\ref{eq27})$ hold.
\end{theorem}

\begin{proof}
We only prove (c). The proof of the statements (a) and (b) is similar.

Define $\wid{H}_{med}:=H_{med}I_0 B_{1/2}$ and let $\wid{H}^*_{med}$ be the adjoint of $\wid{H}_{med}$.
Then, by Remark \ref{rem1} and Lemmas \ref{lem7} and \ref{lem10} we know that
$\wid{F}$ has the factorization $\wid{F}=\wid{H}^*_{med}T_{med}\wid{H}_{med}$
and that $\wid{H}^*_{med}=B^*_{1/2}I^*_0 H^*_{med}$ is bounded from $L^2(D)$ to $L^2(\Sp^1)$ and compact
with dense range in $L^2(\Sp^1)$.
Thus, from the Range Identity \cite[Theorem 1.1]{KL14} and Lemma \ref{lem10} it follows that
the operator $\wid{F}_{\#}$ is positive and $\mathcal{R}(\wid{H}^*_{med})=\mathcal{R}(\wid{F}_{\#}^{{1}/{2}})$.
On the other hand, by Lemma \ref{lem7} and Remark \ref{rem1} we have that
$B^*_{1/2}$ is bijective (and so boundedly invertible) from  $H^{-1/2}(\Sp^1)$ to $L^2(\Sp^1)$ and
$I^*_0$ is injective from $L^2(\Sp^1)$ to $H^{-1/2}(\Sp^1)$.
Thus it is easy to deduce that for any $z\in\R^2$, $\phi_z\in\mathcal{R}(H^*_{med})$ if and only if
$B^*_{1/2}\phi_z=B^*_{1/2}I^*_0\phi_z\in\mathcal{R}(\wid{H}^*_{med})$.
Consequently, by the above argument and (b) of Lemma \ref{lem10} we obtain that
$z\in D$ if and only if $B^*_{1/2}\phi_{z}\in\mathcal{R}(\wid{F}_{\#}^{{1}/{2}})$.
Finally, by Picard's theorem \cite[Theorem A.54]{K11} and the fact that the operator $\wid{F}_{\#}$
is positive, the statement (\ref{eq27}) follows. The proof is thus complete.
\end{proof}

\begin{remark}\label{rem2} {\rm
Define $(\wid{N}^{PW}_{R})_{\#}:=|\Rt(\wid{N}^{PW}_{R})|+|\I(\wid{N}^{PW}_{R})|$.
Then, by Theorem \ref{thm2} and the inequality in \cite[pp. 30]{L08}
we obtain that for any $\alpha\in({1}/{2},1)$ and $R$ large enough,
\be\label{eq35}
\left\|(\wid{N}^{PW}_R)_{\#}-\frac{1}{\sqrt{8 k\pi R}}\wid{F}_\#\right\|_{L^2(\Sp^1)\rightarrow L^2(\Sp^1)}
&\leq& C^{(0)}_\alpha\left\|\wid{N}^{PW}_{R}-\frac{1}{\sqrt{8k\pi R}}\wid{F}
\right\|^{\alpha}_{L^2(\Sp^1)\rightarrow L^2(\Sp^1)}\no\\
&\leq& C^{(1)}_\alpha\left(\frac{1}{R}\right)^{\alpha},
\en
where $C^{(0)}_\alpha$ and $C^{(1)}_\alpha$ are positive constants depending on $\alpha$ but not on $R$.
This implies that the leading order term of the operator $(\wid{N}^{PW}_{R})_{\#}$
is $(1/\sqrt{8 k\pi R})\wid{F}_\#$ in the linear space $\mathcal{L}(L^2(\Sp^1), L^2(\Sp^1))$
for $R$ large enough. On the other hand, by Theorem \ref{thm3} the obstacle $D$ can be recovered by using
the factorization of the operator $\wid{F}_\#$.
Thus, it is expected that if $R$ is large enough then the location and shape of the obstacle $D$
can be approximately recovered by using the indicator function $W(z)$ given in (\ref{eq27}) with
$\wid{F}_\#$ replaced by $(\wid{N}^{PW}_{R})_{\#}$.
Based on these discussions, a numerical algorithm for our inverse problem will be proposed
in details in the next section. Note that the algorithm is based on the factorization method presented
in Theorem \ref{thm3} and the approximate formula given in (\ref{eq35}) and thus called the approximate
factorization method.
}
\end{remark}

\section{Numerical implementation of the approximate factorization method}\label{sec4}
\setcounter{equation}{0}

This section is devoted to the numerical implementation of the approximate factorization method.
Note that the operator $\wid{N}^{PW}_{R}$ in (\ref{eq37}) can not be numerically calculated
since the operators $B_{1/2}$ and $B^*_{1/2}$ are represented as infinite series.
Thus, in order to give a numerical implementation of the method,
we introduce the truncated operator of $\wid{N}^{PW}_{R}$:
\be\label{eq32}
\wid{N}^{PW}_{R,M}:= e^{-i(kR+\pi/4)}B^*_{1/2,M} N^{PW}_{R} B_{1/2,M},
\en
where $M\in\mathbb{N}$, $B_{1/2,M}$ is the truncated operator of $B_{1/2}$ given by
\ben
B_{1/2,M}\varphi:=\sum^{M}_{m=-M}(1+m^2)^{-1/4} a_m\varphi_m
\enn
for $\varphi\in L^2(\Sp^1)$ with the coefficients $a_m$ given in (\ref{eq19})
and $B^*_{1/2,M}$ is the truncated operator of $B^*_{1/2}$ given by
\ben
B^*_{1/2,M}\psi:=\sum\limits^{M}_{m=-M}(1+m^2)^{-1/4}b_m\varphi_m
\enn
for $\psi\in H^{-1/2}(\Sp^1)$ with the coefficients $b_m$ given in (\ref{eq20}).
For the truncated operators $B_{1/2,M}$ and $B^*_{1/2,M}$ we have the following lemma.

\begin{lemma}\label{lem11}
For $M\in\mathbb{N}$ the following assertions hold.
\begin{enumerate}[(a)]
\item
$\|B_{1/2,M}\|_{L^2(\Sp^1)\rightarrow H^{1/2}(\Sp^1)}\le\|B_{1/2}\|_{L^2(\Sp^1)\rightarrow H^{1/2}(\Sp^1)}$.
\item
$\|B^*_{1/2,M}\|_{H^{-1/2}(\Sp^1)\rightarrow L^2(\Sp^1)}\le\|B^*_{1/2}\|_{H^{-1/2}(\Sp^1)\rightarrow L^2(\Sp^1)}$.
\item
For any $r\in\mathbb{N}$, $\|B^*_{1/2}-B^*_{1/2,M}\|_{H^r(\Sp^1)\rightarrow L^2(\Sp^1)}\leq (1+M)^{-(r+1/2)}$.
\end{enumerate}
\end{lemma}

\begin{proof}
Assertions (a) and (b) follow easily from a direct calculation. Thus we only need to prove the assertion (c).
Let $\varphi\in H^r(\Sp^1)$ be of the form (\ref{eq19}). Then we have
\ben
\|B^*_{1/2}\varphi-B^*_{1/2,M}\varphi\|^2_{L^2(\Sp^1)}
&=&\sum\limits_{|m|\geq M+1}|(1+m^2)^{-1/4}a_m|^2\\
&\le& (1+M)^{-2(r+1/2)}\sum\limits_{|m|\geq M+1}(1+m^2)^r|a_m|^2\\
&\le& (1+M)^{-2(r+1/2)}\|\varphi\|^2_{H^r(\Sp^1)}.
\enn
This completes the proof of the lemma.
\end{proof}

Using Theorem \ref{thm1} and Lemma \ref{lem11}, we can obtain the following theorem for the truncated
operator $\wid{N}^{PW}_{R,M}$ and the modified far-field operator $\wid{F}$.

\begin{theorem}\label{thm4}
Let $M\in\mathbb{N}$ and let $R>0$ be large enough.
Then, for any $r\in\mathbb{N}$ there exists a constant $C_r>0$ independent of $R$ such that
\ben
\left\|\wid{N}^{PW}_{R,M}-\frac{1}{\sqrt{8k\pi R}}\wid{F}\right\|_{L^2(\Sp^1)\rightarrow L^2(\Sp^1)}
\le\frac{C_r}{R^{1/2}}\left(\frac{1}{R^{1/2}}+\frac{1}{(1+M)^{r+1/2}}\right).
\enn
\end{theorem}

\begin{proof}
Let $\wid{F}_M$ be the truncated operator of $\wid{F}$ defined by
\ben
\wid{F}_M:=B^*_{1/2,M}F B_{1/2,M},
\enn
where $F$ is the far-field operator given in (\ref{eq34}).
Arbitrarily fix $r\in\mathbb{N}$. Then, by Lemma \ref{lem11} we have
\ben
&&\big\|\widetilde{F}_M-\widetilde{F}\big\|_{L^2(\Sp^1)\rightarrow L^2(\Sp^1)}\\
&&\le\big\|B^*_{1/2,M}F(B_{1/2,M}-B_{1/2})\big\|_{L^2(\Sp^1)\rightarrow L^2(\Sp^1)}
+\big\|(B^*_{1/2,M}-B^*_{1/2})F B_{1/2}\big\|_{L^2(\Sp^1)\rightarrow L^2(\Sp^1)}\\
&&=\big\|(B^*_{1/2,M}-B^*_{1/2})F^* B_{1/2,M}\big\|_{L^2(\Sp^1)\rightarrow L^2(\Sp^1)}
+\big\|(B^*_{1/2,M}-B^*_{1/2})F B_{1/2}\big\|_{L^2(\Sp^1)\rightarrow L^2(\Sp^1)}\\
&&\le\big\|B^*_{1/2,M}-B^*_{1/2}\big\|_{H^r(\Sp^1)\rightarrow L^2(\Sp^1)}
\big\|F^*\big\|_{H^{1/2}(\Sp^1)\rightarrow H^r(\Sp^1)}\big\|B_{1/2,M}\big\|_{L^2(\Sp^1)\rightarrow H^{1/2}(\Sp^1)}\\
&&\;\;+\big\|B^*_{1/2,M}-B^*_{1/2}\big\|_{H^r(\Sp^1)\rightarrow L^2(\Sp^1)}\big\|F\big\|_{H^{1/2}(\Sp^1)\rightarrow H^r(\Sp^1)}\big\|B_{1/2}\big\|_{L^2(\Sp^1)\rightarrow H^{1/2}(\Sp^1)}\\
&&\le\frac{\big\|B_{1/2}\big\|_{L^2(\Sp^1)\rightarrow H^{1/2}(\Sp^1)}} {(1+M)^{(r+1/2)}}\left(\big\|F^*\big\|_{H^{1/2}(\Sp^1)\rightarrow H^r(\Sp^1)}
+\big\|F\big\|_{H^{1/2}(\Sp^1)\rightarrow H^r(\Sp^1)}\right)\\
&&\le {C^{(1)}_r}{(1+M)^{-(r+1/2)}},
\enn
where $C^{(1)}_r>0$ is a constant depending on $r$ but not on $M$ and we have made use of Lemma \ref{lem7}
and the fact that $u^\infty(\hat{x},d)$ is analytic, respectively, in $\hat{x}\in\Sp^1$ and $d\in\Sp^1$
to obtain the last inequality. Then, by the above inequality, Theorem \ref{thm1} and Lemmas \ref{lem7}
and \ref{lem11} we obtain that
\ben
&&\big\|\wid{N}^{PW}_{R,M}-(1/\sqrt{8k\pi R})\wid{F}\big\|_{L^2(\Sp^1)\rightarrow L^2(\Sp^1)}\\
&&\le\big\|\wid{N}^{PW}_{R,M}-(1/\sqrt{8k\pi R})\wid{F}_M\big\|_{L^2(\Sp^1)\rightarrow L^2(\Sp^1)}
+(1/\sqrt{8k\pi R})\big\|\wid{F}_M-\wid{F}\big\|_{L^2(\Sp^1)\rightarrow L^2(\Sp^1)}\\
&&\le\left\|B^*_{1/2,M}\left(N^{PW}_{R}-e^{i(\pi/4+kR)}(1/\sqrt{8k\pi R})F\right)
B_{1/2,M}\right\|_{L^2(\Sp^1)\rightarrow L^2(\Sp^1)}\\
&&\;\;\;+(1/\sqrt{8k\pi R})\big\|\wid{F}_M-\wid{F}\big\|_{L^2(\Sp^1)\rightarrow L^2(\Sp^1)}\\
&&\le\frac{C_r}{R^{1/2}}\left(\frac{1}{R^{1/2}}+\frac{1}{(1+M)^{r+1/2}}\right).
\enn
The proof is thus complete.
\end{proof}

\begin{remark}\label{rem3} {\rm
Define $(\wid{N}^{PW}_{R,M})_{\#}:=|\Rt(\wid{N}^{PW}_{R,M})|+|\I(\wid{N}^{PW}_{R,M})|$.
For simplicity, we will choose $M\geq R$ in the remaining part of this paper.
Then it follows from Lemma \ref{lem11}, Theorem \ref{thm4} and the inequality on
\cite[pp. 30]{L08} that for any $\alpha\in(1/2,1)$, $r\in\mathbb{N}$ and $R$ large enough,
\be\label{eq39}
\left\|(\wid{N}^{PW}_{R,M})_{\#}-\frac{1}{\sqrt{8k\pi R}}\wid{F}_\#\right\|_{L^2(\Sp^1)\rightarrow L^2(\Sp^1)}
&\le& C^{(0)}_\alpha\left\|\wid{N}^{PW}_{R,M}-\frac{1}{\sqrt{8k\pi R}}
\wid{F}\right\|^{\alpha}_{L^2(\Sp^1)\rightarrow L^2(\Sp^1)}\no\\
&\le&\frac{C^{(1)}_{\alpha,r}}{R^{\alpha/2}}\left(\frac{1}{R^{1/2}}+\frac{1}{(1+M)^{r+1/2}}\right)^\alpha\no\\
&\le& C^{(2)}_{\alpha,r}\frac{1}{R^\alpha}
\en
and
\be\label{eq40}
\|B^*_{1/2,M}\phi_z-B^*_{1/2}\phi_z\|_{L^2(\Sp^1)}
&\le& (1+M)^{-(r+1/2)}\|\phi_z\|_{H^r(\Sp^1)}\no\\
&\le& (1+R)^{-(r+1/2)}\|\phi_z\|_{H^r(\Sp^1)},
\en
where $C^{(0)}_\alpha$, $C^{(1)}_{\alpha,r}$, $C^{(2)}_{\alpha,r}$ are positive constants independent
of $M$ and $R$, and $\phi_z$ is the function defined in (\ref{eq23}).
Based on Theorem \ref{thm3} and the approximate formulas (\ref{eq39}) and (\ref{eq40}),
we can define the indicator function
\ben
W^{PW}_M(z):=\left[\sum^\infty_{j=1}{\left|\left(B^*_{1/2,M}\phi_{z},\psi_{j}
\right)_{L^{2}\left(\Sp^{1}\right)}\right|^{2}}\Big/{\lambda_j}\right]^{-1},
\enn
where $\{\lambda_j;\psi_j\}_{j\in\mathbb{N}}$ is an eigensystem of the self-adjoint operator
$(\wid{N}^{PW}_{R,M})_{\#}$. From the above discussion and Theorem \ref{thm3},
it is expected that if $M\geq R$ and $R$ is sufficiently large, then the indicator function $W^{PW}_M(z)$
has a very similar property as $W(z)$ defined in (\ref{eq27}). Thus it is expected that the obstacle $D$
can be numerically recovered by using the discrete form of the indicator function $W^{PW}_M(z)$
(see the formula (\ref{eq29})). This is indeed confirmed by the numerical examples carried out
in Section \ref{sec5}. It should be pointed out that we are currently not able to give a rigorous
theoretical analysis on the property of the indicator function $W^{PW}_M(z)$.
}
\end{remark}

We now give the detailed numerical implementation of the approximate factorization method.
The measured phaseless {\color{HW}total-field} data are obtained as $|u(x_i,\hat{x}_j)|$ with $x_i\in\pa B_{R}$
and $\hat{x}_j\in\Sp^1$, $1\leq i,j\leq L$, where $x_i=R\hat{x}_i$ and $\hat{x}_i$ are uniformly
distributed points on $\Sp^1$. Accordingly, by the trapezoidal rule, the operators $B_{1/2,M}$
and $B^*_{1/2,M}$ can be approximated as follows:
\ben
(B_{1/2,M}\varphi)(\hat{x})&=&(B^*_{1/2,M}\varphi)(\hat{x})\\
&=&\sum\limits_{m=-M}^M\left[(1+m^2)^{-1/4}\left(\int_{\Sp^1}
\varphi(d)\ov{\varphi_m(d)}ds(d)\right)\varphi_m(\hat{x})\right]\\
&\approx&\frac{2\pi}{L}\sum\limits^M_{m=-M}\left[(1+m^2)^{-1/4}\left(\sum\limits^{L}_{j=1}\varphi(\hat{x}_j)
\ov{\varphi_m(\hat{x}_j)}\right)\varphi_m(\hat{x})\right].
\enn
By a direct calculation the approximate values of $(B_{1/2,M}\varphi)(\hat{x})$ and
$(B^*_{1/2,M}\varphi)(\hat{x})$ at the points $\hat{x}_j$, $j=1,2,\ldots,L$, can be computed as
\be\label{eq30}
\left(
  \begin{array}{c}
    (B_{1/2,M}\varphi)(\hat{x}_1) \\
    (B_{1/2,M}\varphi)(\hat{x}_2) \\
    \vdots \\
    (B_{1/2,M}\varphi)(\hat{x}_L) \\
  \end{array}
\right)
=
\left(
  \begin{array}{c}
    (B^*_{1/2,M}\varphi)(\hat{x}_1) \\
    (B^*_{1/2,M}\varphi)(\hat{x}_2) \\
    \vdots \\
    (B^*_{1/2,M}\varphi)(\hat{x}_L) \\
  \end{array}
\right)
\approx{\bf B}_{L,M}
\left(
  \begin{array}{c}
    \varphi(\hat{x}_1) \\
    \varphi(\hat{x}_2) \\
    \vdots \\
    \varphi(\hat{x}_L) \\
  \end{array}
\right).
\en
Here, ${\bf B}_{L,M}$ is a complex symmetric matrix defined by
${\bf B}_{L,M}:=\frac{2\pi}{L}{\bf C}_{L,M}{\bf D}_M{\bf C}^*_{L,M}$, where
${\bf C}_{L,M}=(c_{ij})_{1\leq i\leq L,1\leq j\leq 2M+1}$ with $c_{ij}=\varphi_{j-(M+1)}(\hat{x}_i)$
and ${\bf D}_M$ is a diagonal matrix given by ${\bf D}_M={\rm Diag}(d_1,d_2,\ldots,d_{2M+1})$ with $d_j=[1+(j-M-1)^2]^{-1/4}$.
In particular, the approximate values of $(B^*_{1/2,M}\phi_z)(\hat{x})$ at the points $\hat{x}_j$,
$j=1,2,\ldots,L$, can be obtained from (\ref{eq30}) with $\varphi$ replaced by $\phi_z$.
Similarly, by the trapezoidal rule again, the approximate values of $(N^{PW}_{R}\varphi)(\hat{x})$
at the points $\hat{x}_j$, $j=1,2,\ldots,L$, can be computed as
\be\label{eq31}
\left(
  \begin{array}{c}
    (N^{PW}_{R}\varphi)(\hat{x}_1) \\
    (N^{PW}_{R}\varphi)(\hat{x}_2) \\
    \vdots \\
    (N^{PW}_{R}\varphi)(\hat{x}_L) \\
  \end{array}
\right)
\approx\frac{2\pi}{L}{\bf N}_{L}
\left(
  \begin{array}{c}
    \varphi(\hat{x}_1) \\
    \varphi(\hat{x}_2) \\
    \vdots \\
    \varphi(\hat{x}_L) \\
  \end{array}
\right),
\en
where ${\bf N}_{L}=(n_{ij})_{1\leq i,j\leq L}$ with
$n_{ij}=[|u(R\hat{x}_i,\hat{x}_j)|^2-1]e^{ikR\hat{x}_i\cdot\hat{x}_j}$.
Then, by the definition (\ref{eq32}), the approximate formulas (\ref{eq30}) and (\ref{eq31})
in conjunction with a direct calculation the approximate values of $(\wid{N}^{PW}_{R,M}\varphi)(\hat{x})$
at the points $\hat{x}_j$, $j=1,2,\ldots,L$, can be easily obtained:
\be\label{eq33}
\left(
  \begin{array}{c}
    (\wid{N}^{PW}_{R,M}\varphi)(\hat{x}_1) \\
    (\wid{N}^{PW}_{R,M}\varphi)(\hat{x}_2) \\
    \vdots \\
    (\wid{N}^{PW}_{R,M}\varphi)(\hat{x}_L) \\
  \end{array}
\right)
\approx\frac{2\pi}{L}{\bf\wid{N}}_{L,M}
\left(
  \begin{array}{c}
    \varphi(\hat{x}_1) \\
    \varphi(\hat{x}_2) \\
    \vdots \\
    \varphi(\hat{x}_L) \\
  \end{array}
\right),
\en
where
\be\label{eq28}
{\bf\wid{N}}_{L,M}:=e^{-i(kR+\pi/4)}{\bf B}_{L,M}{\bf N}_{L}{\bf B}_{L,M}.
\en
Based on (\ref{eq30}), (\ref{eq33}) and the indicator function $W^{PW}_M(z)$ defined in Remark \ref{rem3},
we introduce the discrete indicator function $W^{PW}_{L,M}(z)$:
\be\label{eq29}
W^{PW}_{L,M}(z):=\left[\sum\limits^{L}_{l=1}\frac{|\wid{\phi}^*_{z,M}\psi_{l,M}|^2}{\lambda_{l,M}}\right]^{-1},
\quad z\in\R^2,
\en
where $\wid{\phi}_{z,M}:={\bf B}_{L,M}(\phi_z(\hat{x}_1),\phi_z(\hat{x}_2),\ldots,\phi_z(\hat{x}_L))^T$
and $\{\lambda_{l,M};\psi_{l,M}\}_{l=1}^L$ is the eigensystem of the complex symmetric matrix
$({\bf\wid{N}}_{L,M})_\#:=|\Rt({\bf\wid{N}}_{L,M})|+|\I({\bf\wid{N}}_{L,M})|$.
Here, the real and imaginary parts of the matrix
${\bf\wid{N}}_{L,M}$ are complex symmetric matrices given by
\ben
\Rt({\bf\wid{N}}_{L,M}):=\frac{1}{2}\left({\bf\wid{N}}_{L,M}
+{\bf\wid{N}}^*_{L,M}\right),\quad
\I({\bf\wid{N}}_{L,M}):=\frac{1}{2i}\left(
{\bf\wid{N}}_{L,M}-{\bf\wid{N}}^*_{L,M}\right),
\enn
respectively. Moreover, the matrices
$|\Rt({\bf\wid{N}}_{L,M})|$ and $|\I({\bf\wid{N}}_{L,M})|$
are the discrete form of the operators $|\Rt(\wid{N}^{PW}_{R,M})|$ and $|\I(\wid{N}^{PW}_{R,M})|$, respectively, which
are also complex symmetric and can be computed as in
\cite[Section 4]{AG05}.
From Theorem \ref{thm3} and the discussion in Remark \ref{rem3}, it is expected that $W^{PW}_{L,M}(z)$
is much bigger for $z\in D$ than that for $z\notin D$ if $M\geq R$ and $R$ is sufficiently large.
Here, the constant $2\pi/L$ in (\ref{eq33}) is not taken into account for the indicator
function (\ref{eq29}) since it does not make any contribution to the numerical algorithm.

The numerical algorithm of the approximate factorization method can be presented as follows.

\begin{algorithm}
Let $K$ be the sampling region which contains the unknown obstacle $D$.
\begin{enumerate}[(1)]
\item
Choose $\mathcal{T}_m$ to be a mesh of $K$. Set $R$ and $M$ to be large numbers with $M\geq R$.
\item
Collect the phaseless {\color{HW}total-field} data $|u(x_i,\hat{x}_j)|$ with $x_i\in\pa B_{R}$ and
$\hat{x}_j\in\Sp^1$, $1\leq i,j\leq L$, generated by the incident plane waves
$u^i(x,\hat{x}_j)=e^{ik x\cdot\hat{x}_j}$, $1\leq j\leq L$.
\item
Compute the matrix ${\bf\wid{N}}_{L,M}$ by using (\ref{eq28}).
\item
For all sampling points $z\in\mathcal{T}_m$, compute the indicator function $W^{PW}_{L,M}(z)$ given in (\ref{eq29}).
\item
Locate all those sampling points $z\in\mathcal{T}_m$ such that $W^{PW}_{L,M}(z)$ takes a large value,
which represent the obstacle $D$.
\end{enumerate}
\end{algorithm}

\section{Numerical examples}\label{sec5}
\setcounter{equation}{0}

In this section, we present several numerical experiments to demonstrate the effectiveness of
our inversion algorithm. To generate the synthetic data, the forward scattering problem is solved
by using the Nystr\"{o}m method \cite{CK13}.
Further, the noisy phaseless {\color{HW}total-field} data $|u_\delta(x,d)|$, $x\in\pa B_{R}$, $d\in\Sp^1$,
are simulated by
\ben
|u_\delta(x,d)|=|u(x,d)|(1+\delta\zeta),
\enn
where $\delta$ is the noise ratio and $\zeta$ is the uniformly distributed random number in $[-1,1]$.

In the following examples, we choose $M=100$. The parametrization of the
test curves for the boundary $\pa D$ are given in Table \ref{table1}, {\color{HW}where $(c_1,c_2)$
denotes the center of the test curves.}

\begin{table}[h]
\centering
\begin{tabular}{ll}
\hline
Type & Parametrization\\
\hline
{\color{HW}Circle} &{\color{HW}$x(t)=(c_1,c_2)+(\cos{t},\sin{t}),\;t\in[0,2\pi]$}\\
%Apple shaped &$(c_1,c_2)+[({0.5+0.4\cos{t}+0.1\sin(2t)})/({1+0.7\cos{t}})](\cos{t},\sin{t}),\;t\in[0,2\pi]$\\
Kite shaped &$x(t)={\color{HW}(c_1,c_2)} + (\cos{t}+0.65\cos(2t)-0.65,1.5\sin{t}),\;t\in[0,2\pi]$\\
Peanut shaped &$x(t)={\color{HW}(c_1,c_2)}+\sqrt{\cos^2{t}+0.25\sin^2{t}}(\cos{t},\sin{t})$, $t\in[0,2\pi]$\\
Rounded square &$x(t)={\color{HW}(c_1,c_2)}+({3}/{4})(\cos^3{t}+\cos{t},\sin^3{t}+\sin{t}),\;t\in[0,2\pi]$\\
Rounded triangle &$x(t)={\color{HW}(c_1,c_2)}+(2+0.3\cos(3t))(\cos{t},\sin{t}),\;t\in[0,2\pi]$\\
\hline
\end{tabular}
\caption{Parametrization of the boundary curves}\label{table1}
\end{table}

\textbf{Example 1.}
We first consider a peanut-shaped, sound-soft obstacle.
See Figure \ref{fig1}(a) for the physical configuration.
{\color{HW}We choose $k=40$, $L=200$ and $R=10$.}
Figure \ref{fig1} presents the reconstruction results of the obstacle
by using the phaseless {\color{HW}total-field} data from incident plane waves without noise,
with 10\% noise and with 20\% noise, respectively.

\begin{figure}[htbp]
\begin{minipage}[t]{0.4\linewidth}
\centering
\includegraphics[width=2.8in]{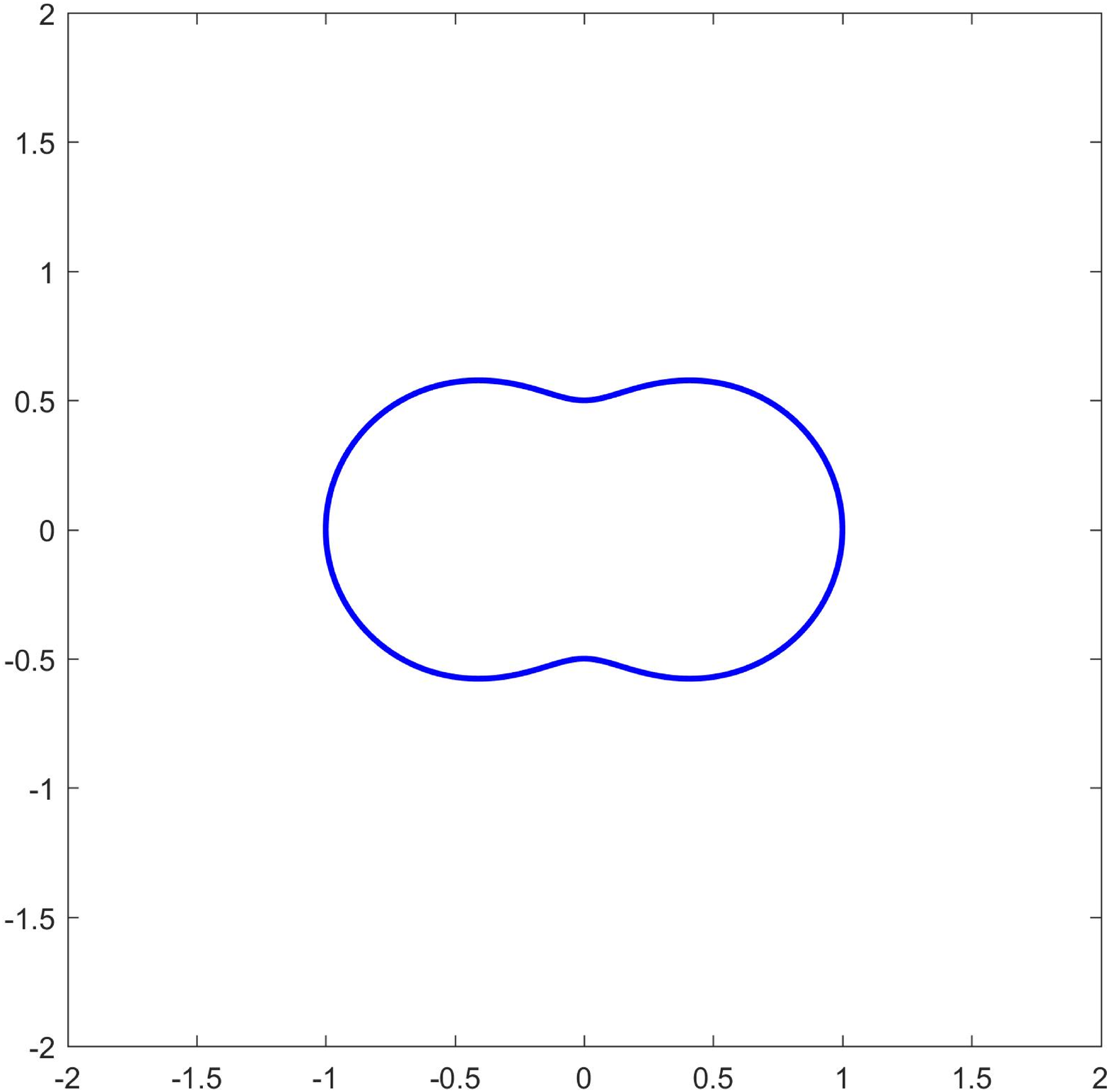}
%\caption{$k=5$, $k_1=1$, $\la=k^2/k_1^2=25>1$}\label{fig1-ex1}
(a) Physical configuration
\end{minipage}\qquad\qquad
\begin{minipage}[t]{0.4\linewidth}
\centering
\includegraphics[width=2.8in]{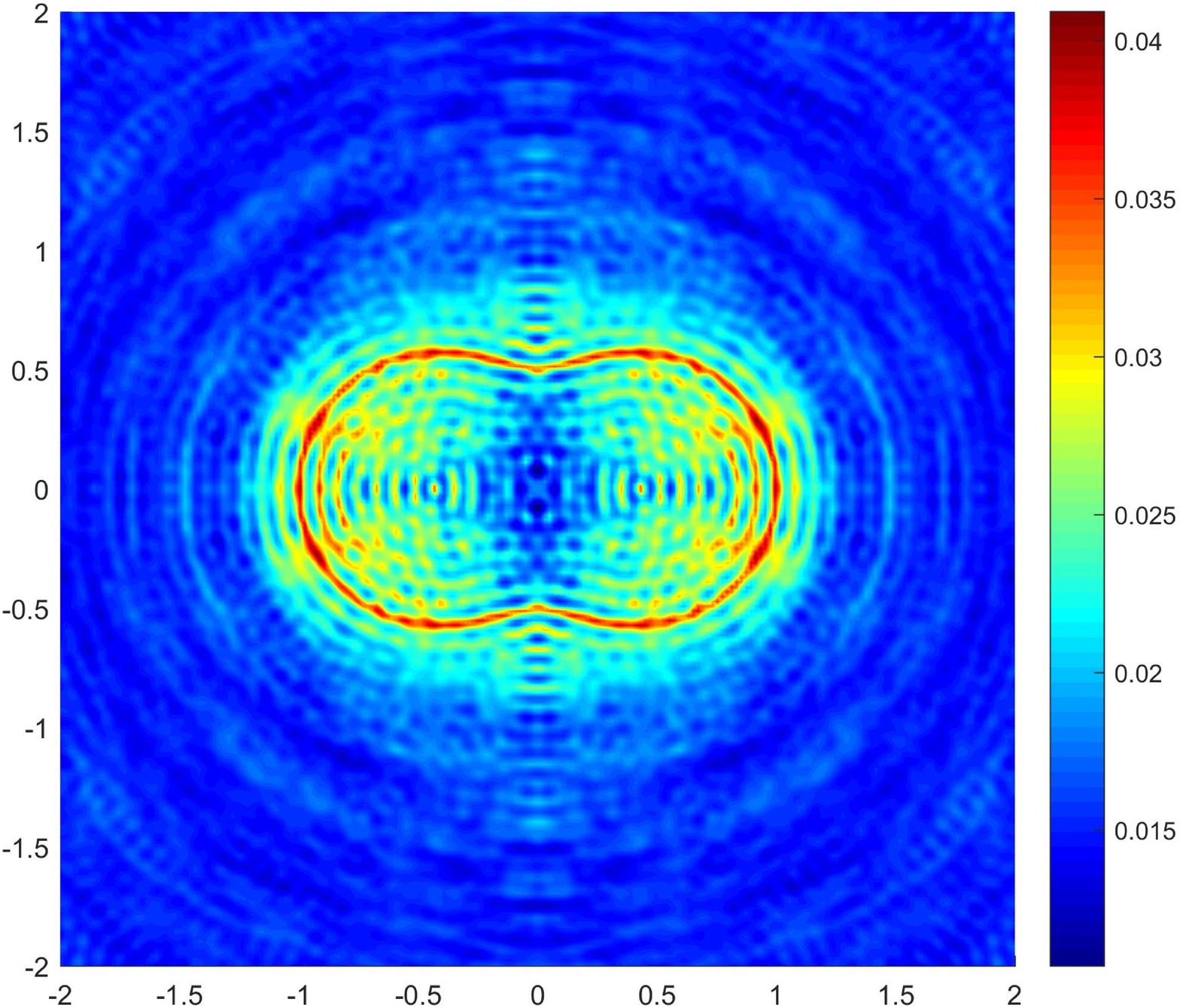}
%\caption{without \textbf{noise}}\label{fig2-ex1}
(b) {\color{HW}$k=40$, $R=10$, no noise}
\end{minipage}
\begin{minipage}[t]{0.4\linewidth}
\centering
\includegraphics[width=2.8in]{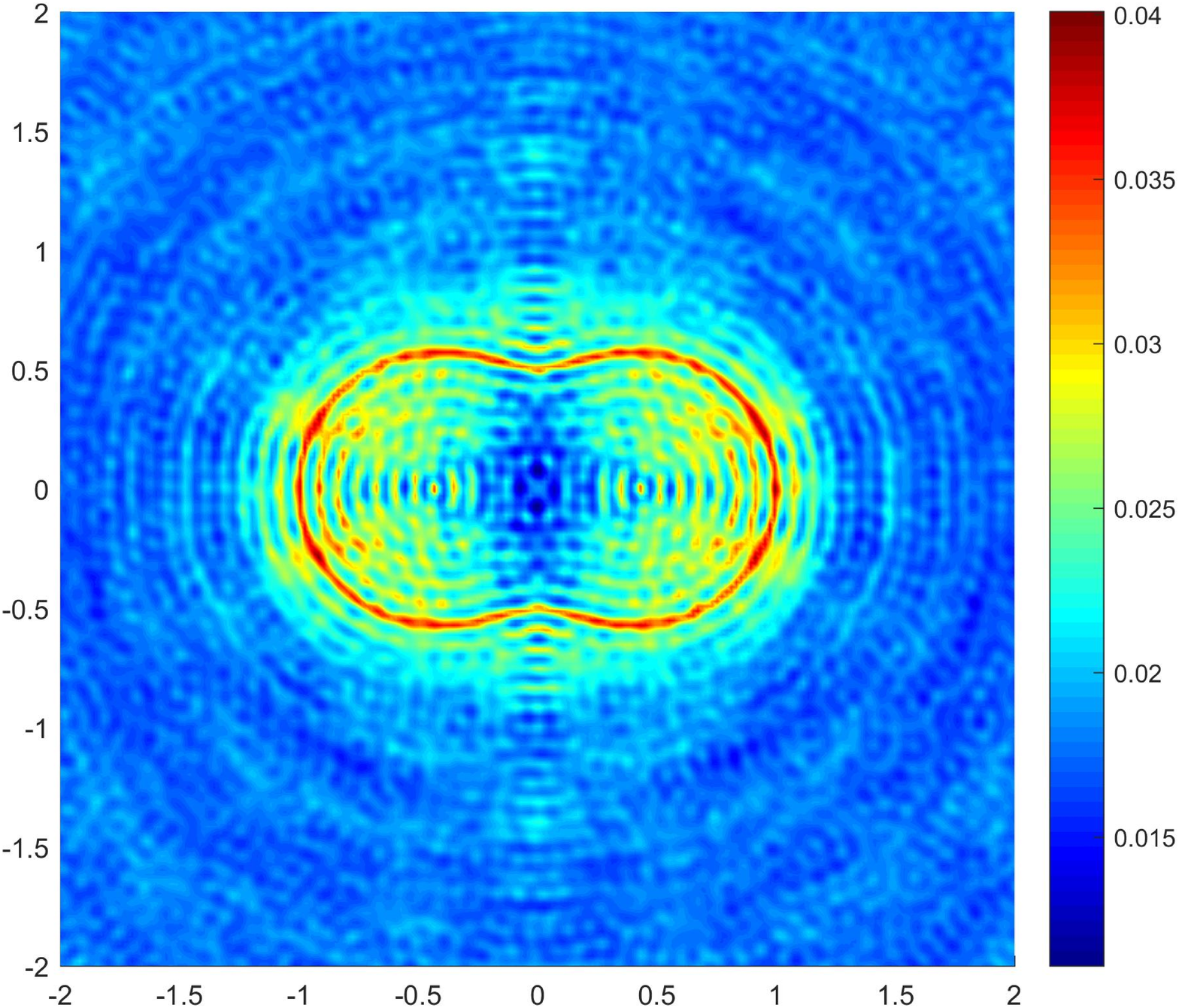}
%\caption{\textbf{2\% noise}}\label{fig3-ex1}
(c) {\color{HW}$k=40$, $R=10$, 10\% noise}
\end{minipage}\qquad\qquad
\begin{minipage}[t]{0.4\linewidth}
\centering
\includegraphics[width=2.8in]{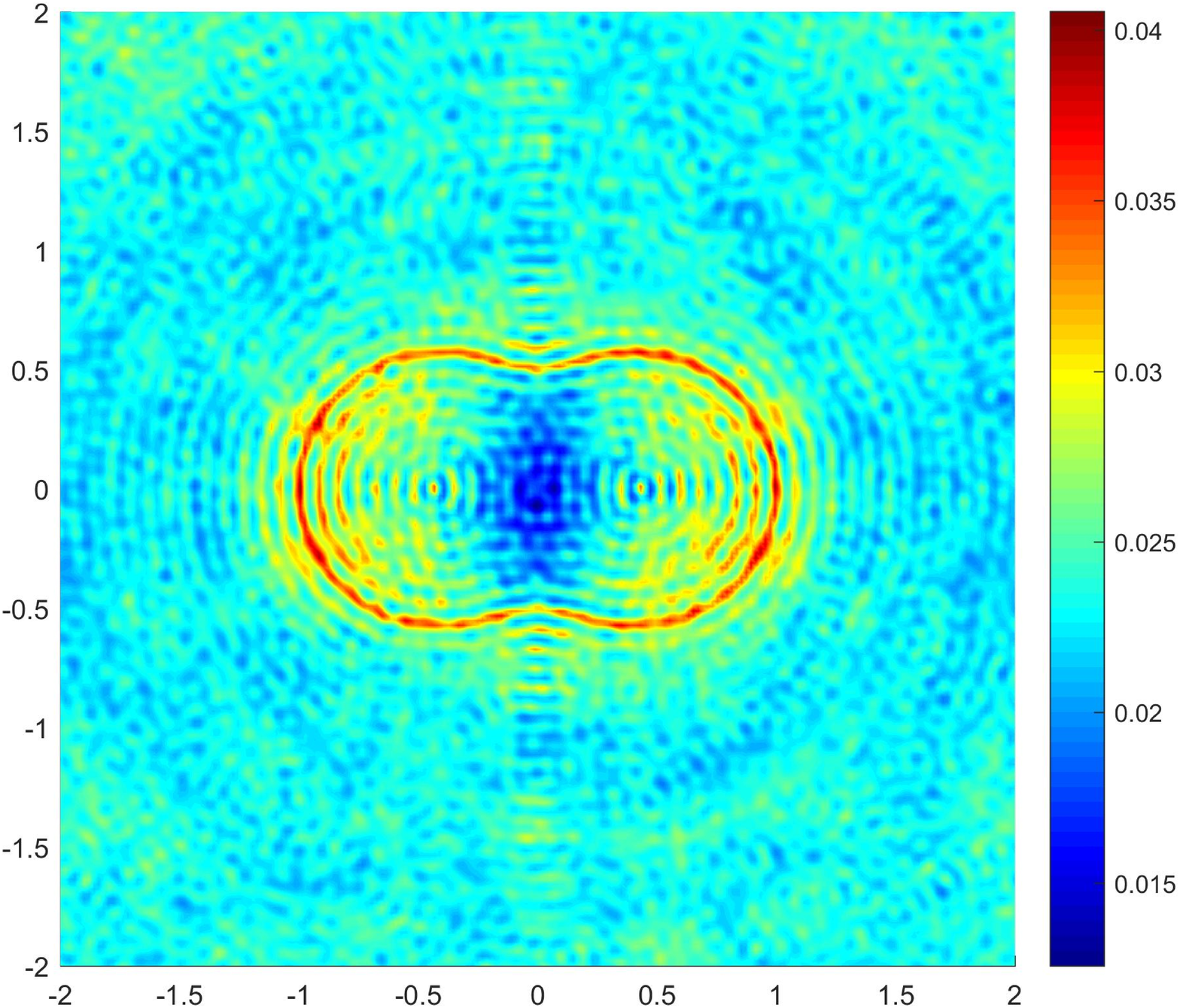}
%\caption{\textbf{5\% noise}}\label{fig4-ex1}
(d) {\color{HW}$k=40$, $R=10$, 20\% noise}
\end{minipage}
\caption{Reconstruction of a peanut-shaped, sound-soft obstacle.
}\label{fig1}
\end{figure}

\textbf{Example 2.}
We now consider a rounded square-shaped, sound-hard obstacle.
See Figure \ref{fig2}(a) for the physical configuration.
{\color{HW}We choose $k=10$, $L=150$ and $R=10$.}
Figure \ref{fig2} presents the reconstruction results of the obstacle
by using the phaseless {\color{HW}total-field} data from incident plane waves without noise,
with 10\% noise and with 20\% noise, respectively.

\begin{figure}[htbp]
\begin{minipage}[t]{0.4\linewidth}
\centering
\includegraphics[width=2.8in]{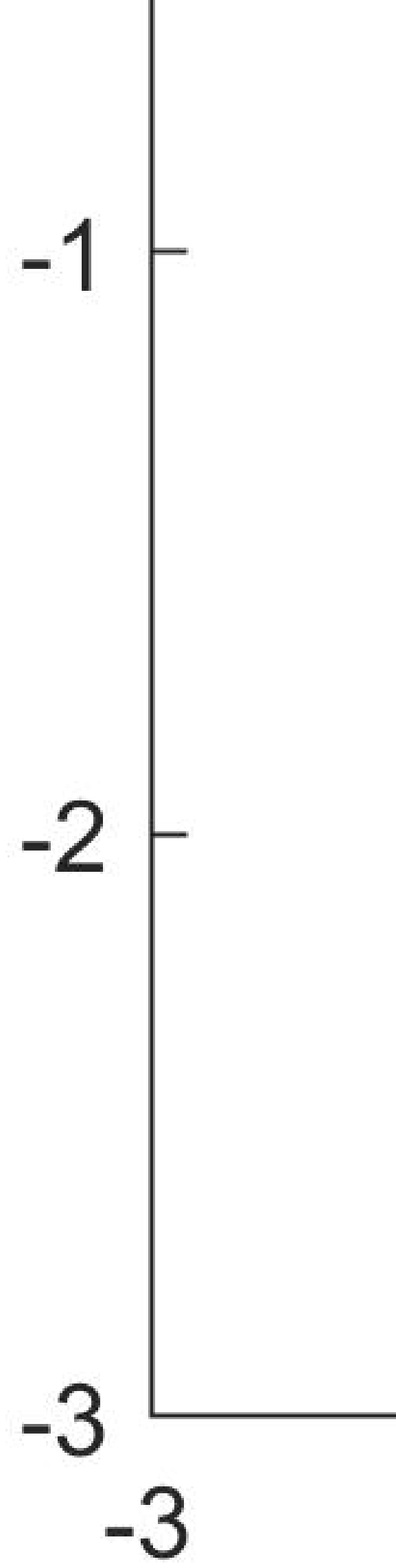}
%\caption{$k=5$, $k_1=1$, $\la=k^2/k_1^2=25>1$}\label{fig1-ex1}
(a) Physical configuration
\end{minipage}\qquad\qquad
\begin{minipage}[t]{0.4\linewidth}
\centering
\includegraphics[width=2.8in]{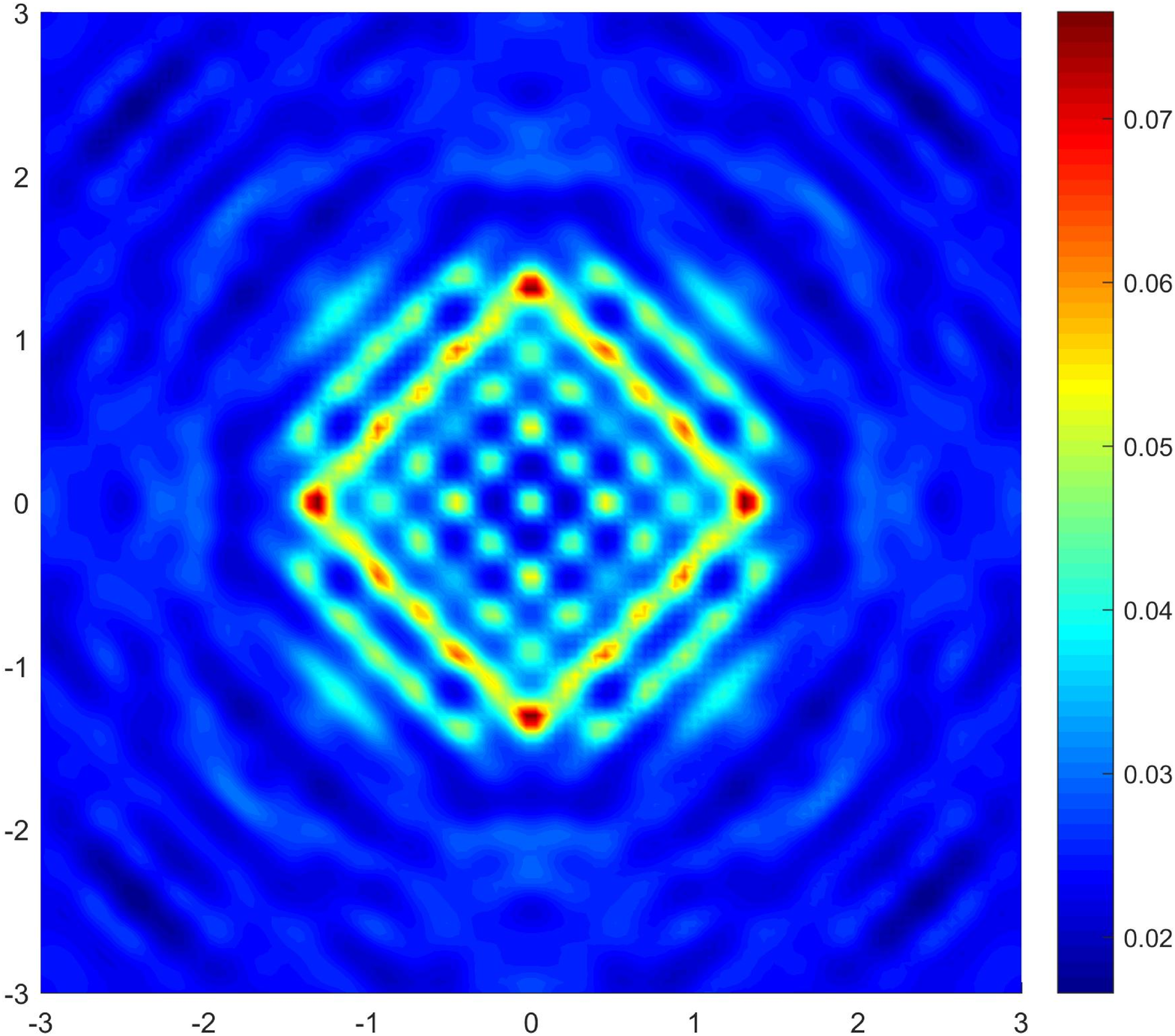}
%\caption{without \textbf{noise}}\label{fig2-ex1}
(b) {\color{HW}$k=10$, $R=10$, no noise}
\end{minipage}
\begin{minipage}[t]{0.4\linewidth}
\centering
\includegraphics[width=2.8in]{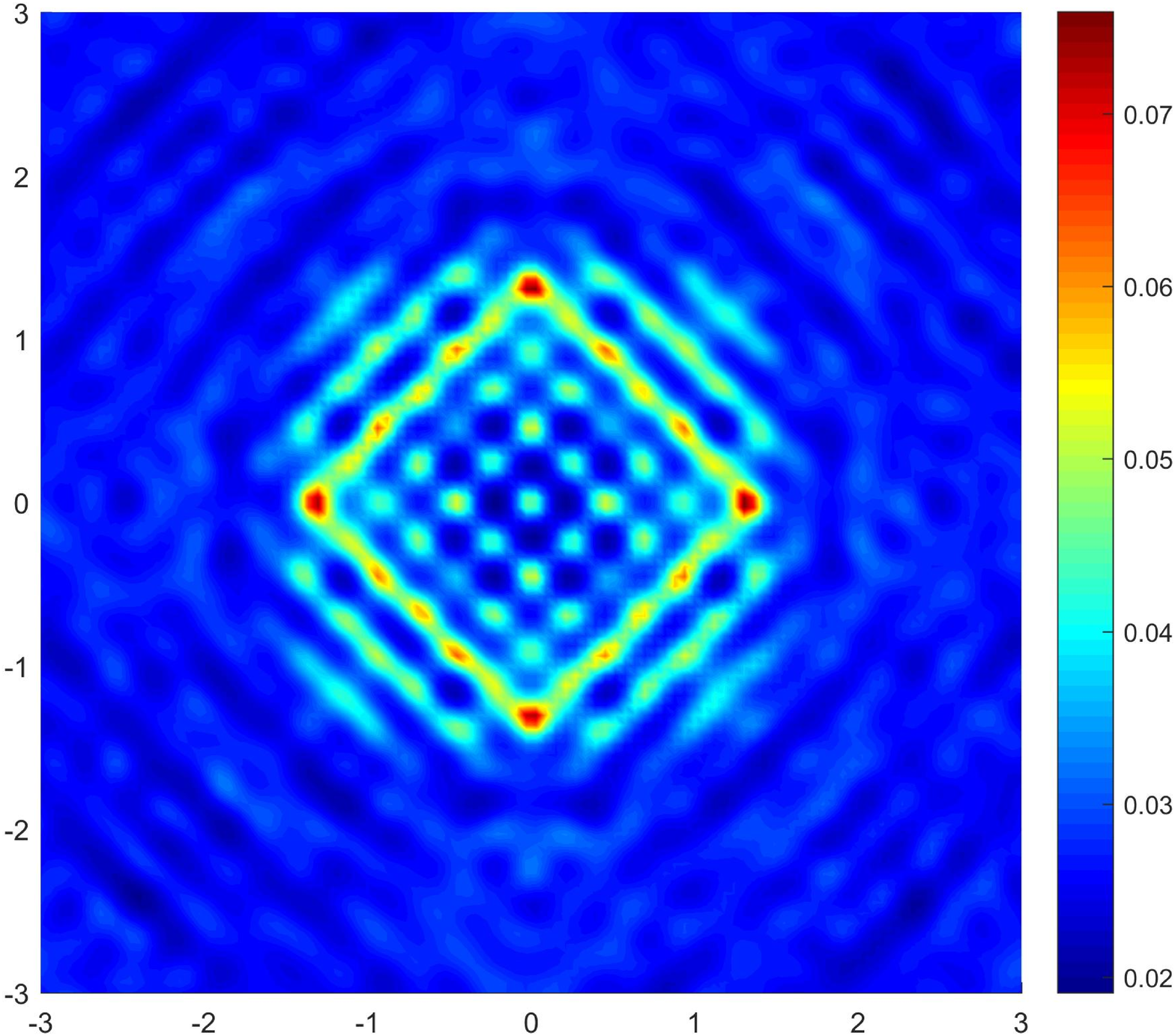}
%\caption{\textbf{2\% noise}}\label{fig3-ex1}
(c) {\color{HW}$k=10$, $R=10$, 10\% noise}
\end{minipage}\qquad\qquad
\begin{minipage}[t]{0.4\linewidth}
\centering
\includegraphics[width=2.8in]{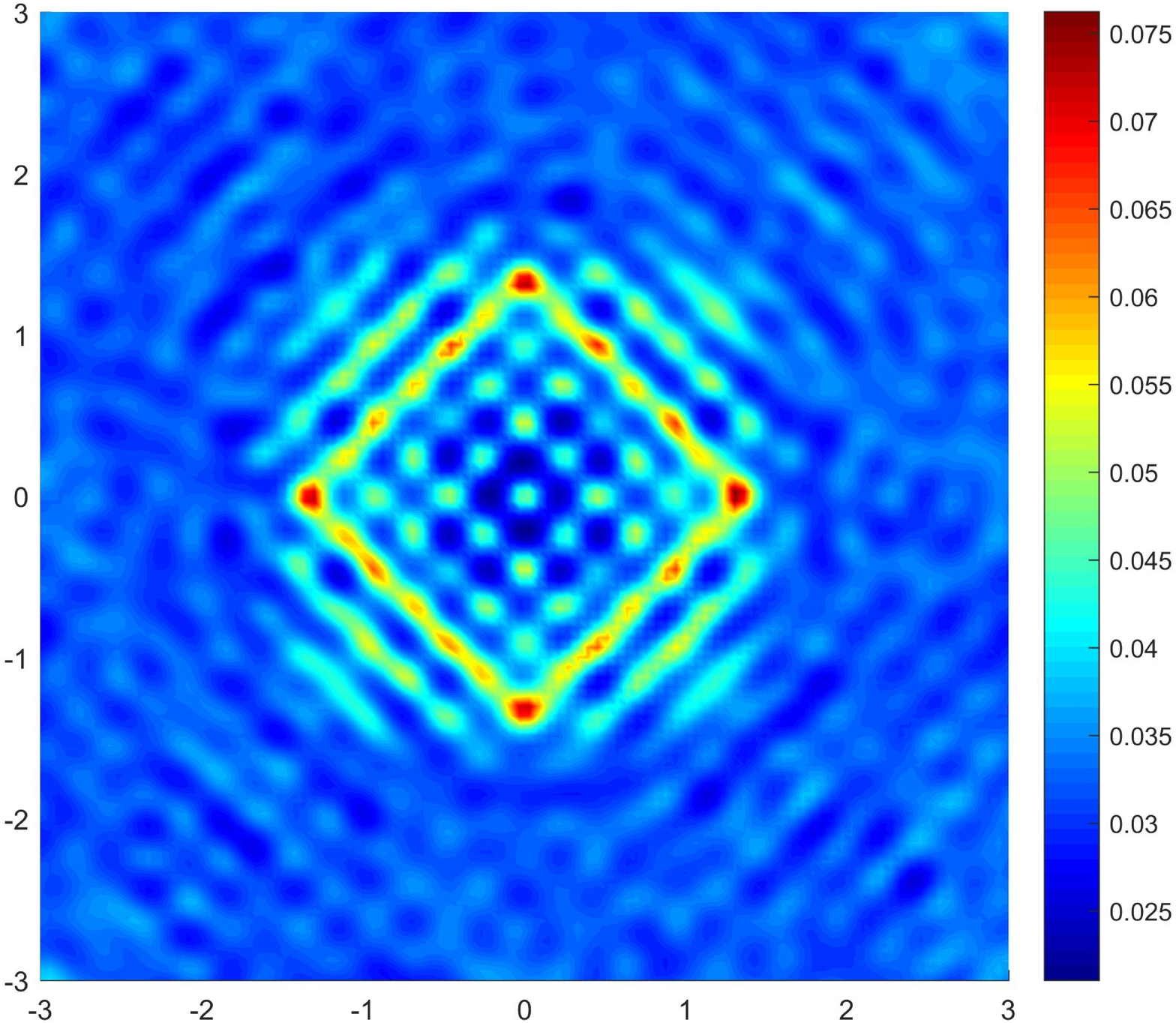}
%\caption{\textbf{5\% noise}}\label{fig4-ex1}
(d) {\color{HW}$k=10$, $R=10$, 20\% noise}
\end{minipage}
\caption{Reconstruction of a rounded square-shaped, sound-hard obstacle.
}\label{fig2}
\end{figure}

\textbf{Example 3.}
This example considers a kite-shaped, impedance obstacle.
The impedance function is given by $\rho(x(t))=(5+5i)*(1+0.5\sin t),t\in[0,2\pi]$,
where $x(t)$ is the parametrization of the boundary $\pa D$.
See Figure \ref{fig3}(a) for the physical configuration.
{\color{HW}
In this example, we investigate the effect of the radius $R$ of the measurement circle $\partial B_R$ on the
imaging results. We choose $k=10$ and $L=150$.
Further, the noise ratio is set to be $\delta=10\%$.}
Figure \ref{fig3} presents the reconstruction results of the obstacle
by using the phaseless {\color{HW}total-field} data from incident plane waves
{\color{HW}with the radius of the measurement circle $\partial B_R$ to be
$R=4$, $8$, $12$, respectively}. {\color{HW}From Figure \ref{fig3} it can be seen that
the reconstruction result is getting better with $R$ getting larger.
This is consistent with the discussions in Remark \ref{rem2}.}

\begin{figure}[htbp]
\begin{minipage}[t]{0.4\linewidth}
\centering
\includegraphics[width=2.8in]{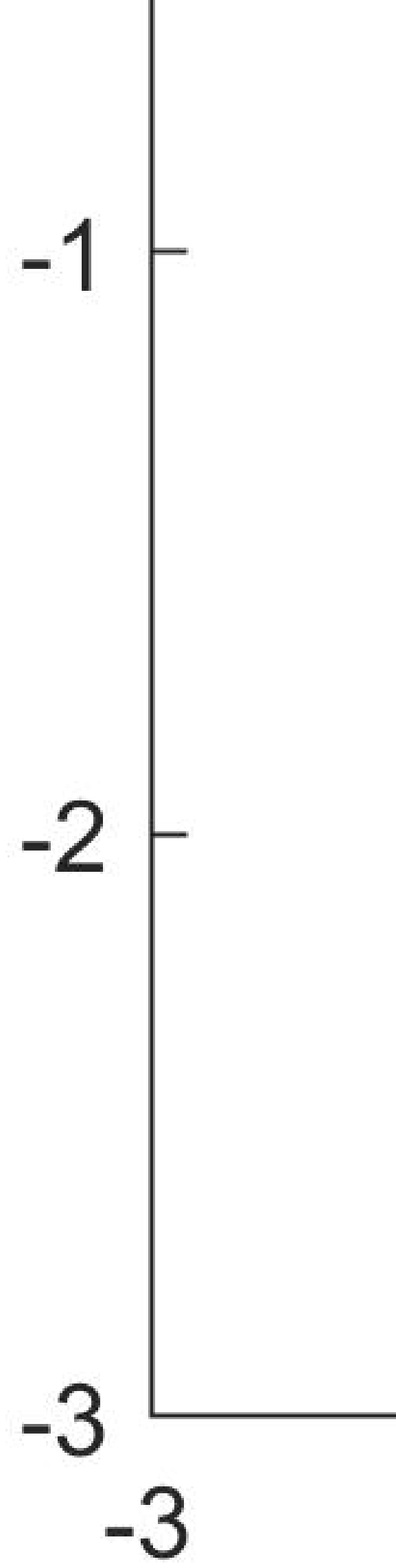}
%\caption{$k=5$, $k_1=1$, $\la=k^2/k_1^2=25>1$}\label{fig1-ex1}
(a) Physical configuration
\end{minipage}\qquad\qquad
\begin{minipage}[t]{0.4\linewidth}
\centering
\includegraphics[width=2.8in]{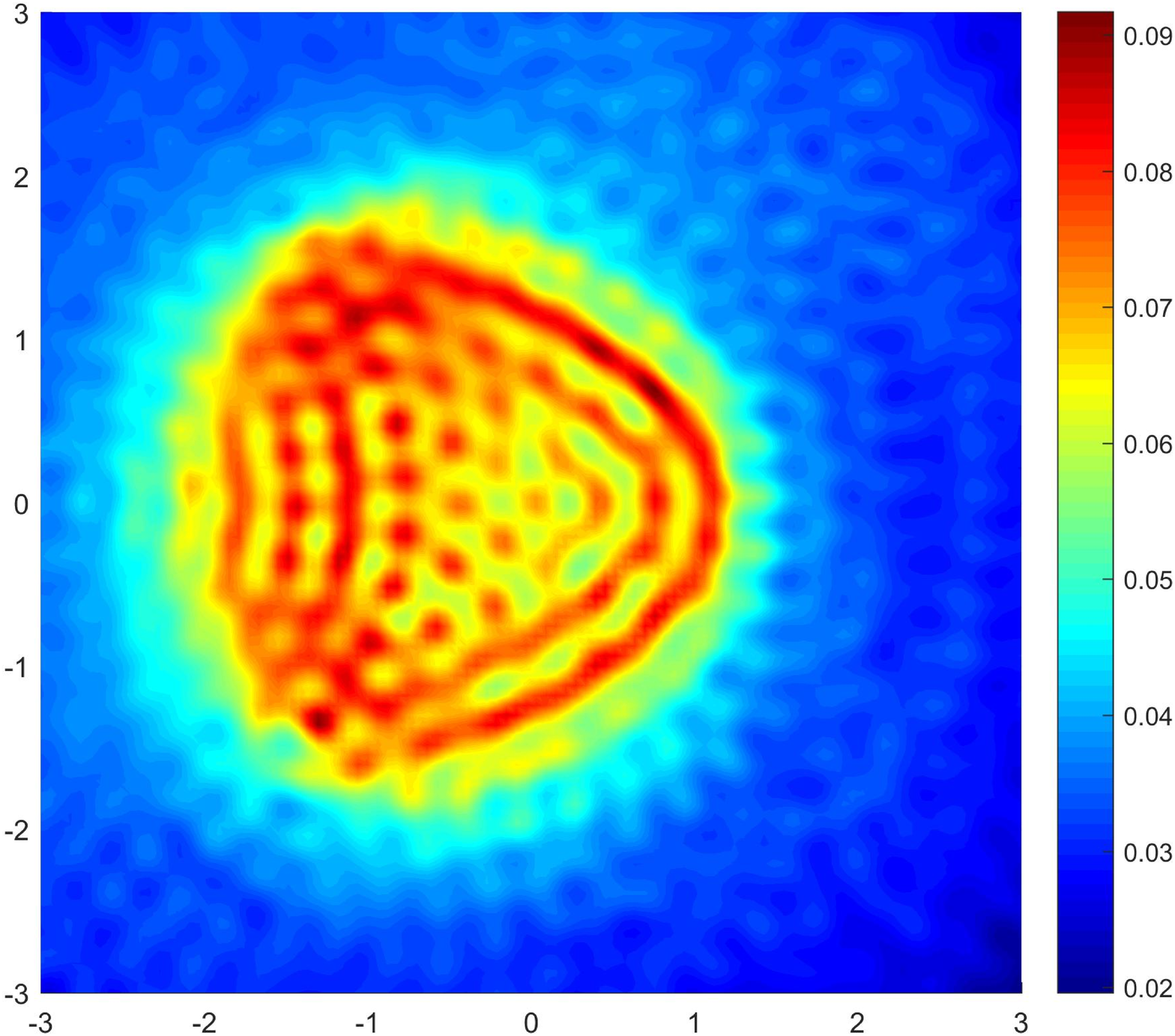}
%\caption{without \textbf{noise}}\label{fig2-ex1}
(b) {\color{HW}$k=10$, $R=4$, 10\% noise}
\end{minipage}
\begin{minipage}[t]{0.4\linewidth}
\centering
\includegraphics[width=2.8in]{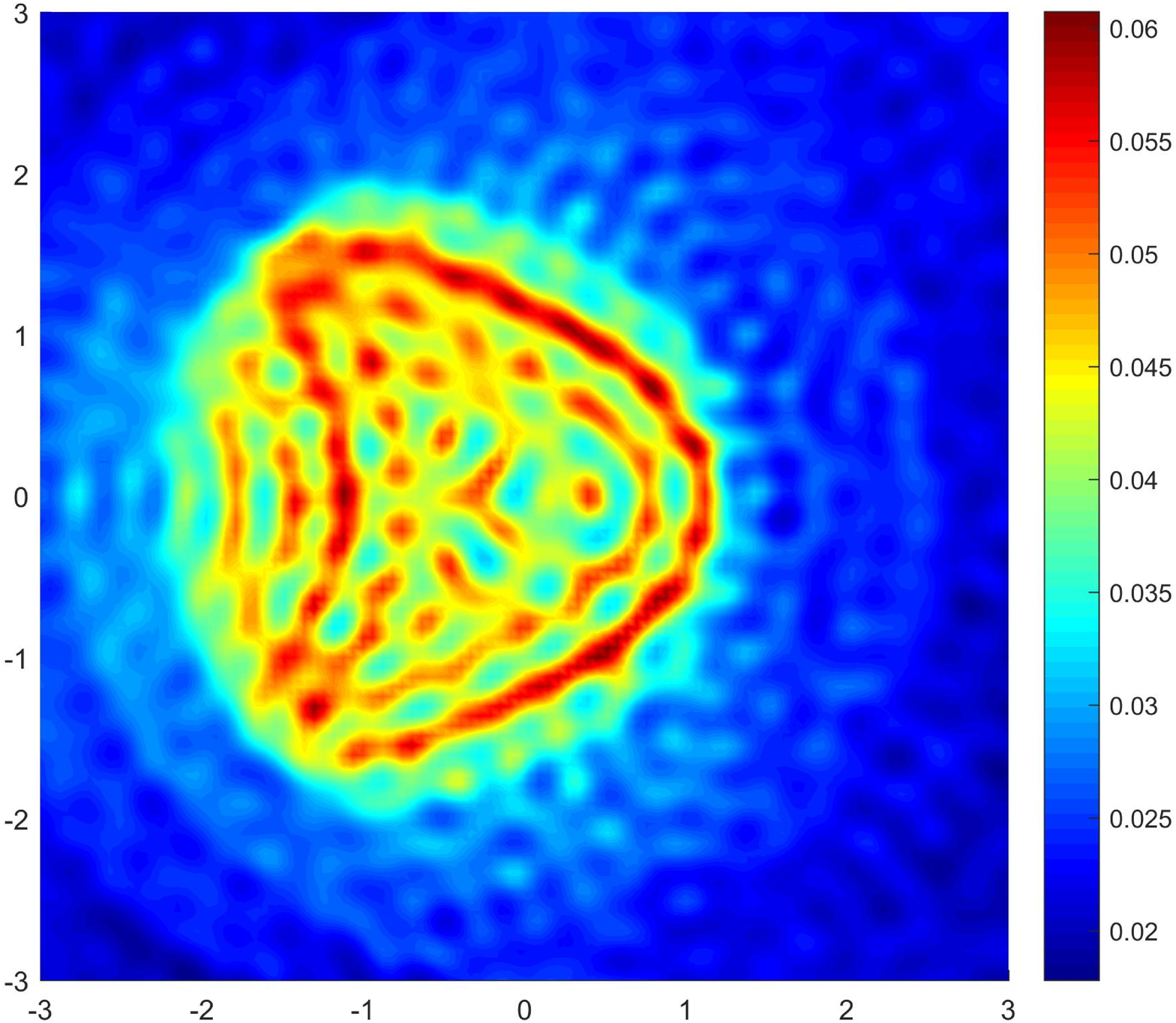}
%\caption{\textbf{2\% noise}}\label{fig3-ex1}
(c) {\color{HW}$k=10$, $R=8$, 10\% noise}
\end{minipage}\qquad\qquad
\begin{minipage}[t]{0.4\linewidth}
\centering
\includegraphics[width=2.8in]{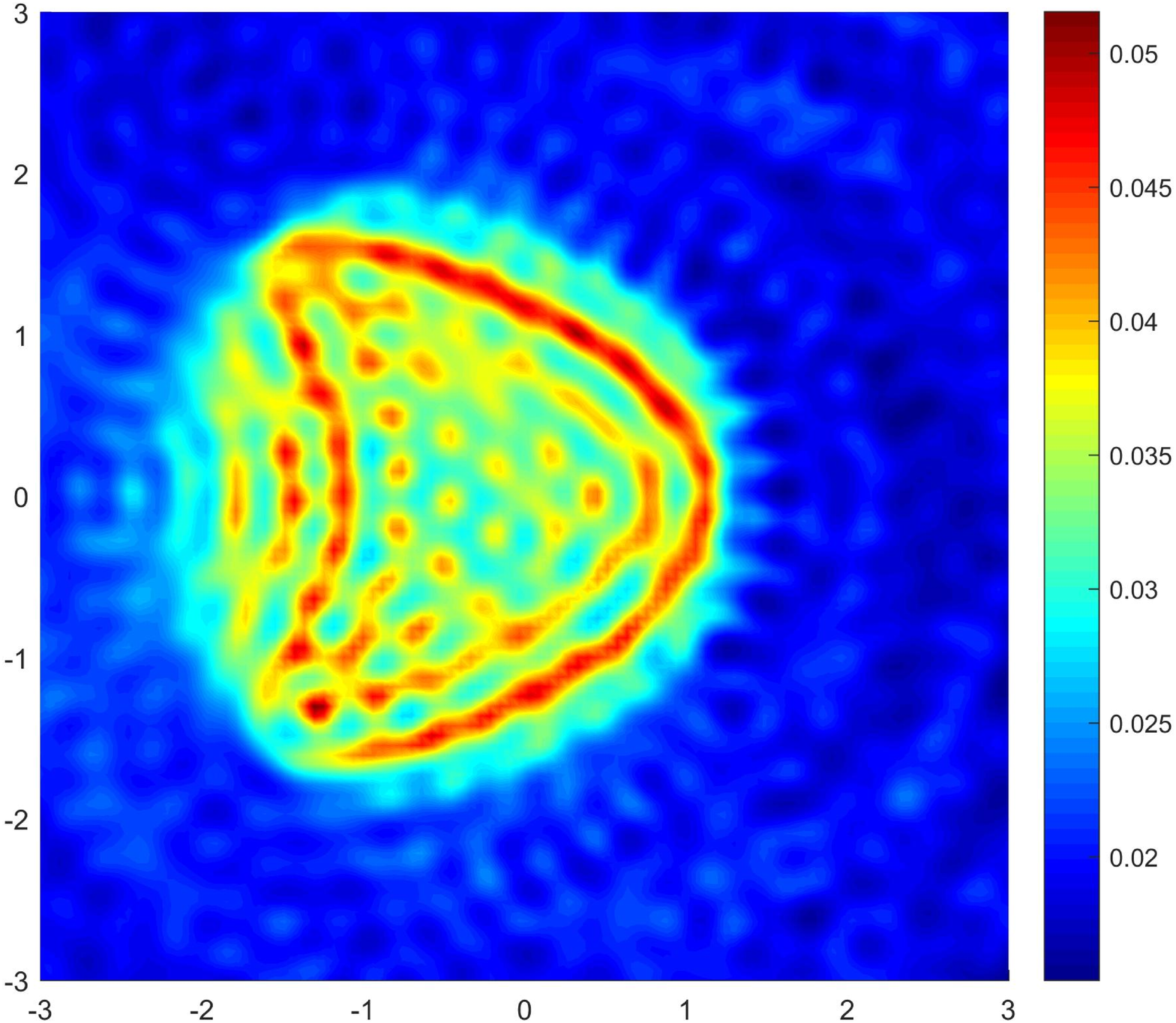}
%\caption{\textbf{5\% noise}}\label{fig4-ex1}
(d) {\color{HW}$k=10$, $R=12$, 10\% noise}
\end{minipage}
\caption{Reconstruction of a kite-shaped, impedance obstacle with the impedance function
$\rho(x(t))=(5+5i)*(1+0.5\sin t),t\in[0,2\pi]$, on $\pa D$.
}\label{fig3}
\end{figure}

\textbf{Example 4.}
This example considers a rounded triangle-shaped, penetrable obstacle.
The refractive index in $D$ is given by $n(x)=2+1.5i$.
{\color{HW}
In this example, we investigate the reconstruction results of the same obstacle with three different locations.
For all three cases, we choose $k=10$, $L=150$ and $R=10$, and the noise ratio is set to be $\delta=10\%$.
In Cases 1, 2 and 3, the obstacle $D$ is centered at $(4,2)$ (Figure \ref{fig4}(a)),
at $(-3,4)$ (Figure \ref{fig4}(c)) and at $(-2,-4)$ (Figure \ref{fig4}(e)), respectively,
with the corresponding reconstruction result of the obstacle in Figure \ref{fig4}(b), (d) and (f).
}

\begin{figure}[htbp]
\begin{minipage}[t]{0.4\linewidth}
\centering
\includegraphics[width=2.8in]{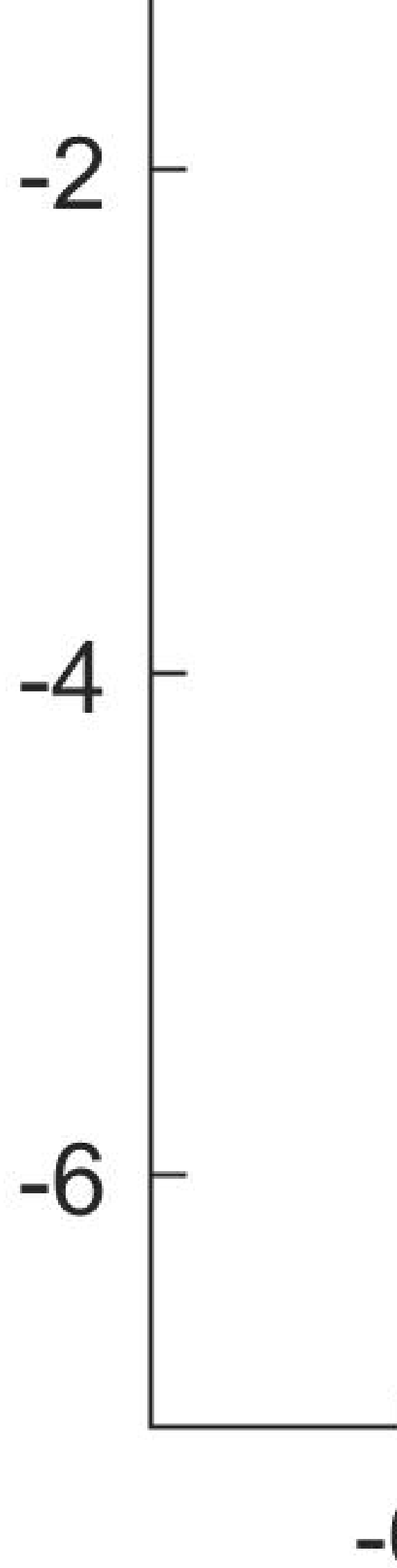}
%\caption{$k=5$, $k_1=1$, $\la=k^2/k_1^2=25>1$}\label{fig1-ex1}
(a) {\color{HW}Case 1: Physical configuration}
\end{minipage}\qquad\qquad
\begin{minipage}[t]{0.4\linewidth}
\centering
\includegraphics[width=2.8in]{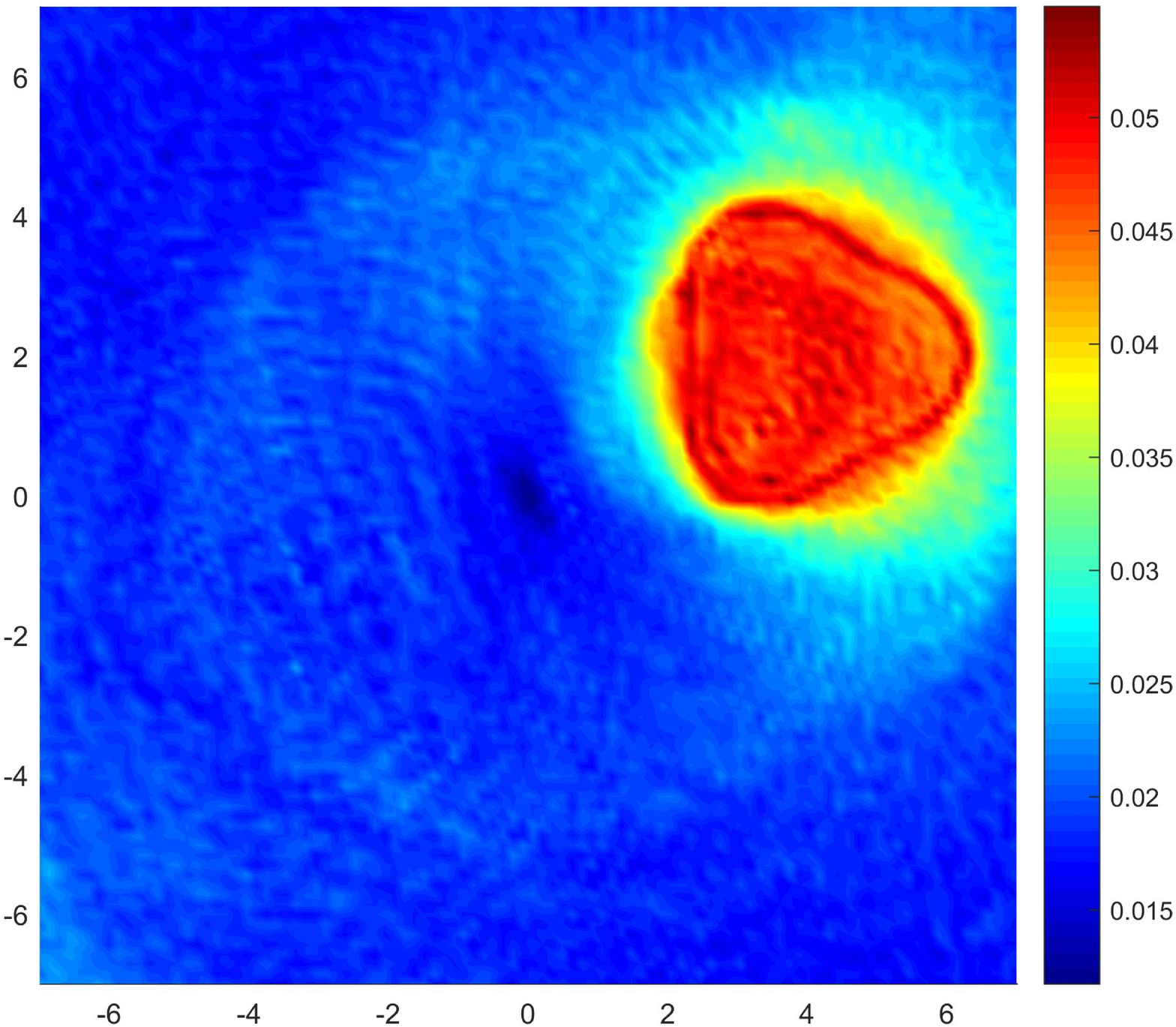}
%\caption{without \textbf{noise}}\label{fig2-ex1}
(b) {\color{HW}Case 1: $k=10$, $R=10$, 10\% noise}
\end{minipage}
\begin{minipage}[t]{0.4\linewidth}
\centering
\includegraphics[width=2.8in]{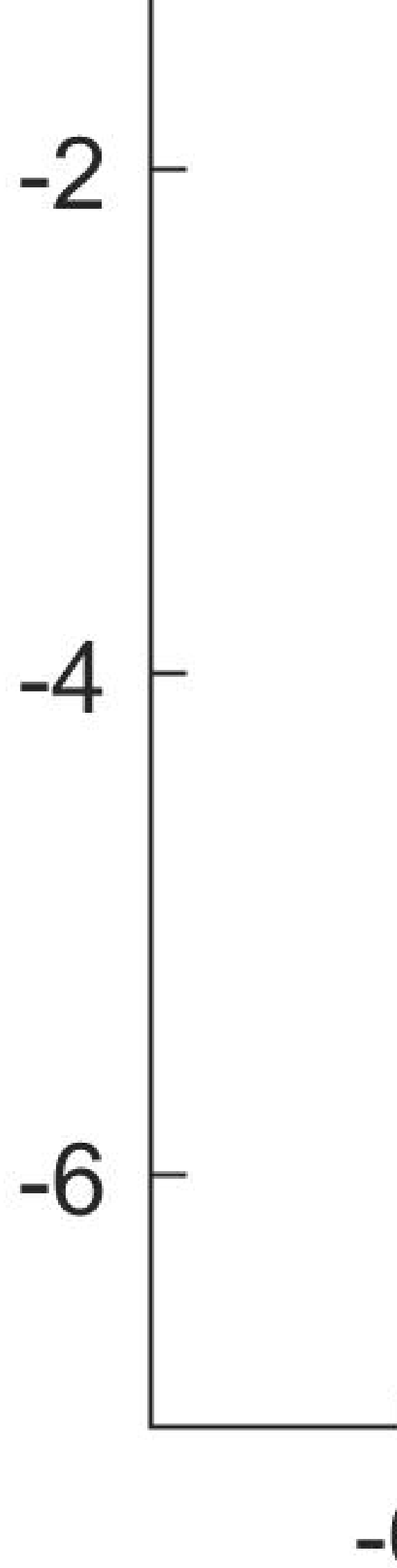}
%\caption{\textbf{2\% noise}}\label{fig3-ex1}
(c) {\color{HW}Case 2: Physical configuration}
\end{minipage}\qquad\qquad
\begin{minipage}[t]{0.4\linewidth}
\centering
\includegraphics[width=2.8in]{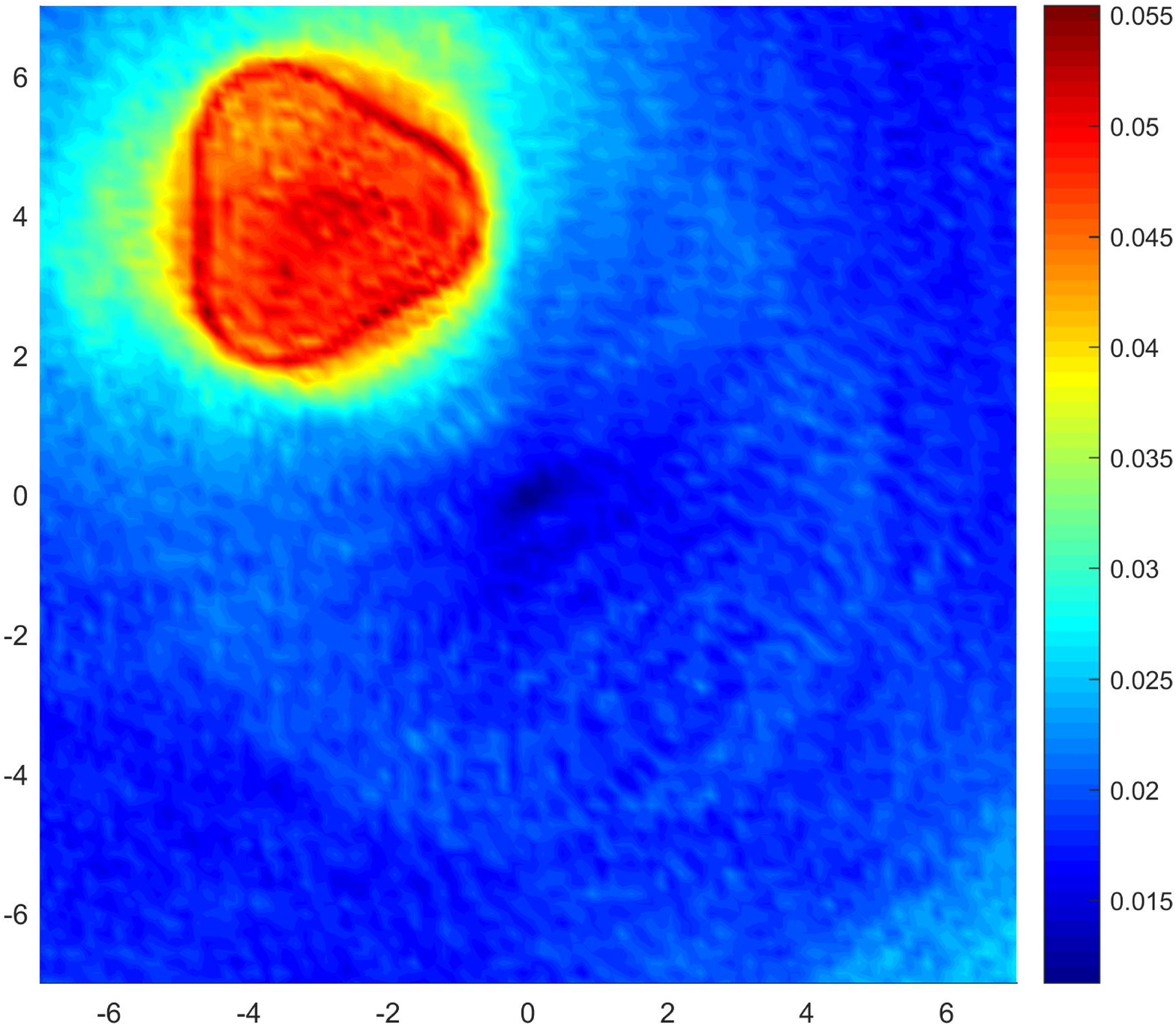}
%\caption{\textbf{5\% noise}}\label{fig4-ex1}
(d) {\color{HW}Case 2: $k=10$, $R=10$, 10\% noise}
\end{minipage}
\begin{minipage}[t]{0.4\linewidth}
\centering
\includegraphics[width=2.8in]{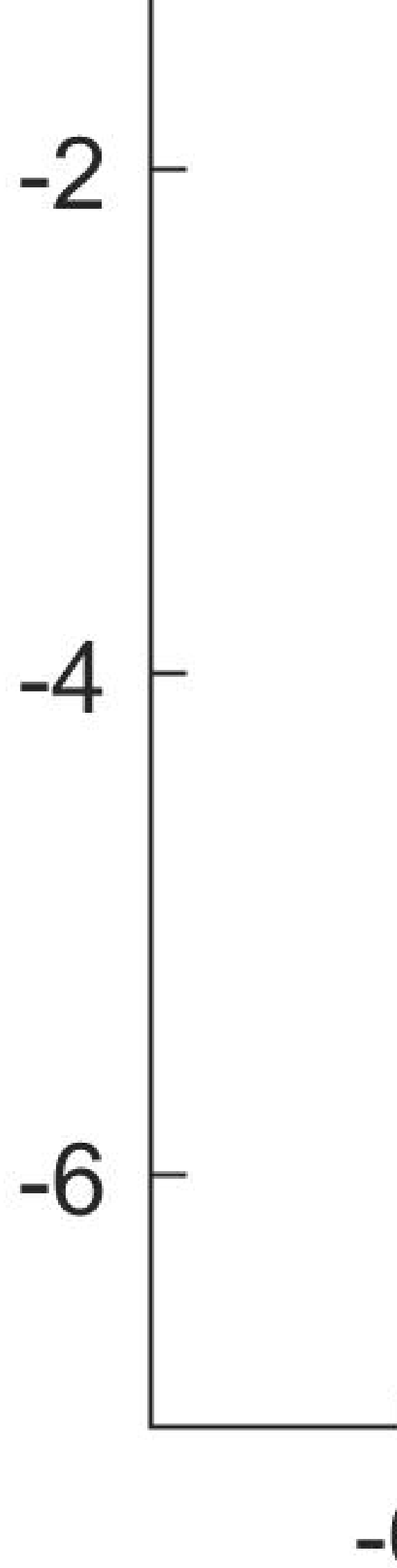}
%\caption{$k=5$, $k_1=1$, $\la=k^2/k_1^2=25>1$}\label{fig1-ex1}
(e) {\color{HW}Case 3: Physical configuration}
\end{minipage}\qquad\qquad
\begin{minipage}[t]{0.4\linewidth}
\centering
\includegraphics[width=2.8in]{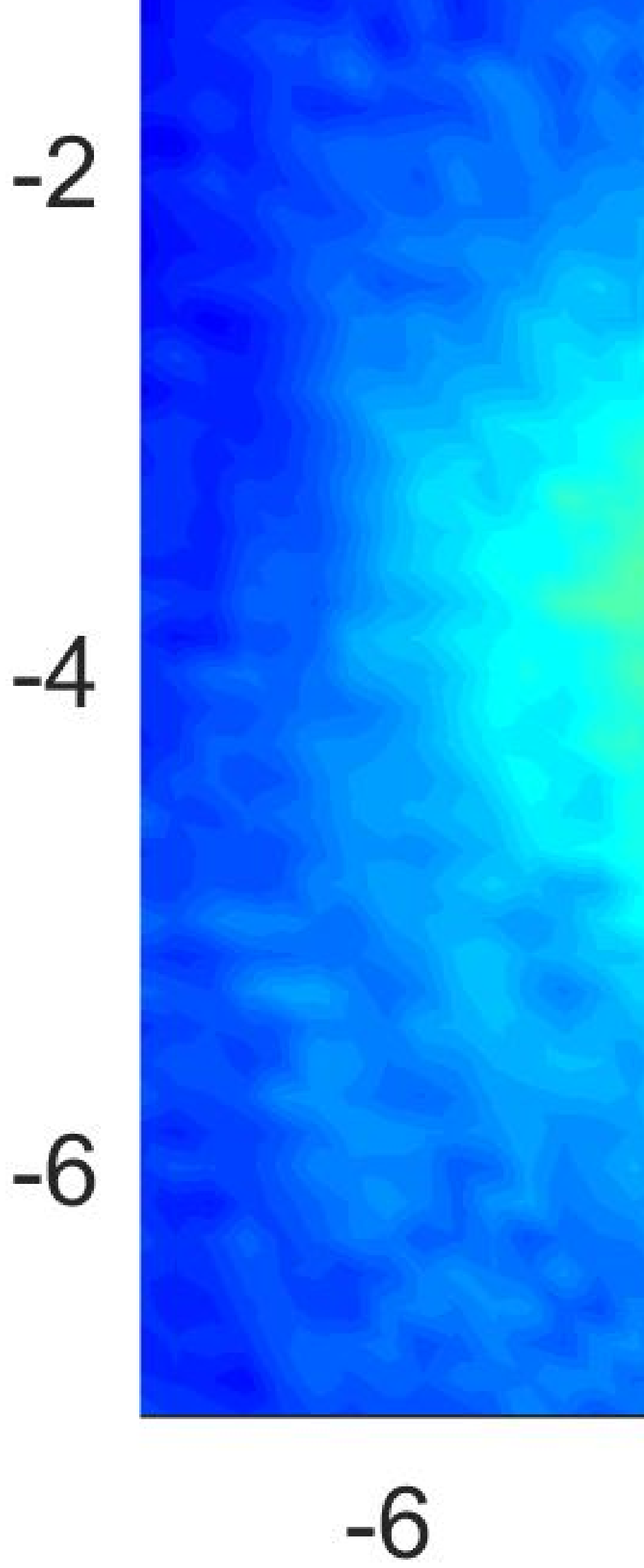}
%\caption{without \textbf{noise}}\label{fig2-ex1}
(f) {\color{HW}Case 3: $k=10$, $R=10$, 10\% noise}
\end{minipage}
\caption{Reconstruction of a rounded triangle-shaped, penetrable obstacle with the refractive index
$n(x)=2+1.5i$ in $D$ at three different locations.
}\label{fig4}
\end{figure}

{\color{HW}
\textbf{Example 5.}
We consider a sound-soft
obstacle $D=D_1\cup D_2$ with two disjoint components $D_1$ and $D_2$, where $D_1$ is
rounded square-shaped and $D_2$ is circle-shaped.
See Figure \ref{fig5}(a) for the physical configuration.
We choose $k=10$, $L=200$ and $R=15$.
Figure \ref{fig5} presents the reconstruction results of the obstacle
by using the phaseless total-field data from incident plane waves without noise,
with 10\% noise and with 20\% noise, respectively.
It is shown in Figure \ref{fig5} that both the location and the shape of the two components $D_1$ and $D_2$
can be numerically recovered with our inversion algorithm.
}

\begin{figure}[htbp]
\begin{minipage}[t]{0.4\linewidth}
\centering
\includegraphics[width=2.8in]{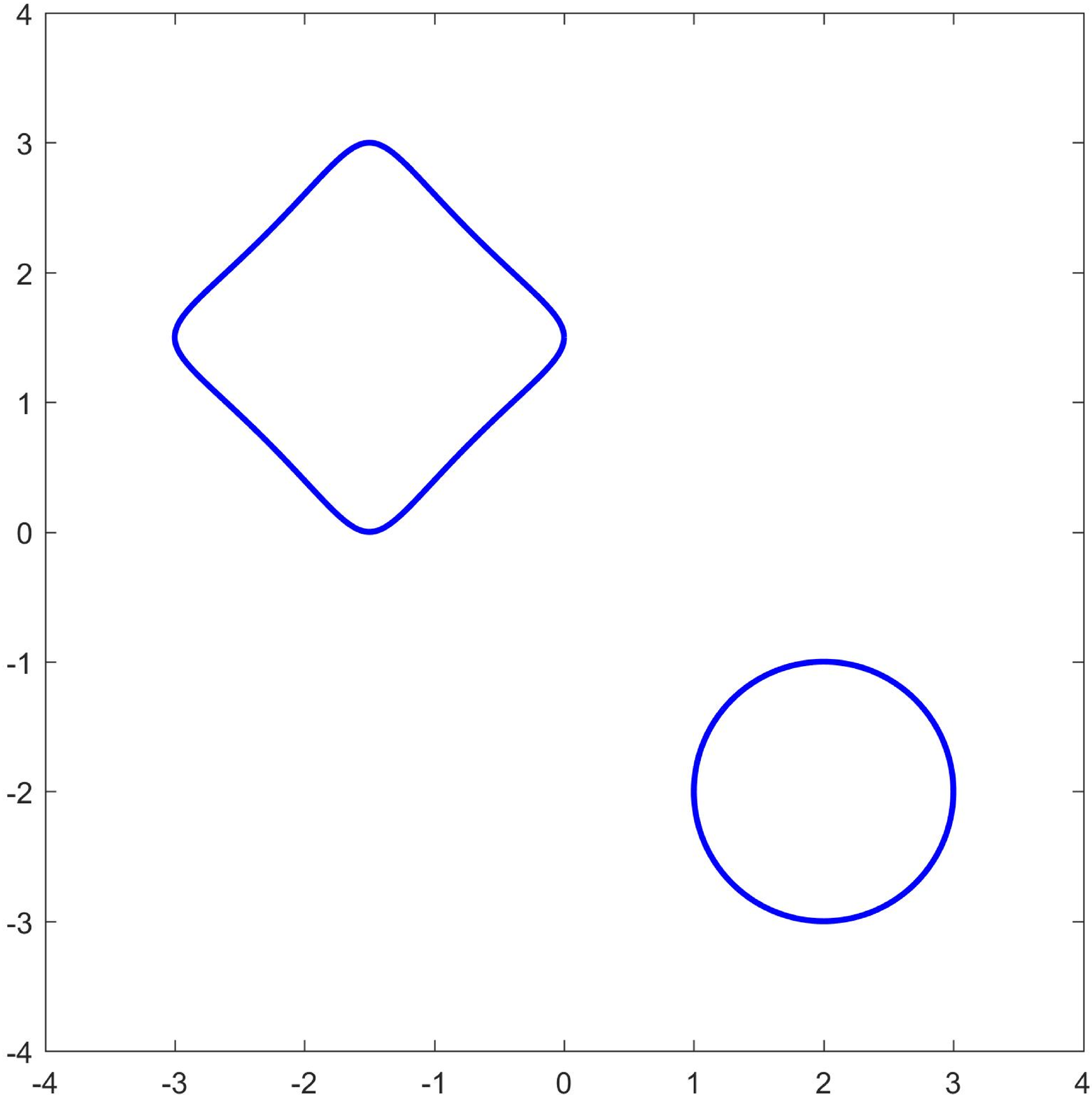}
%\caption{$k=5$, $k_1=1$, $\la=k^2/k_1^2=25>1$}\label{fig1-ex1}
(a) {\color{HW}Physical configuration}
\end{minipage}\qquad\qquad
\begin{minipage}[t]{0.4\linewidth}
\centering
\includegraphics[width=2.8in]{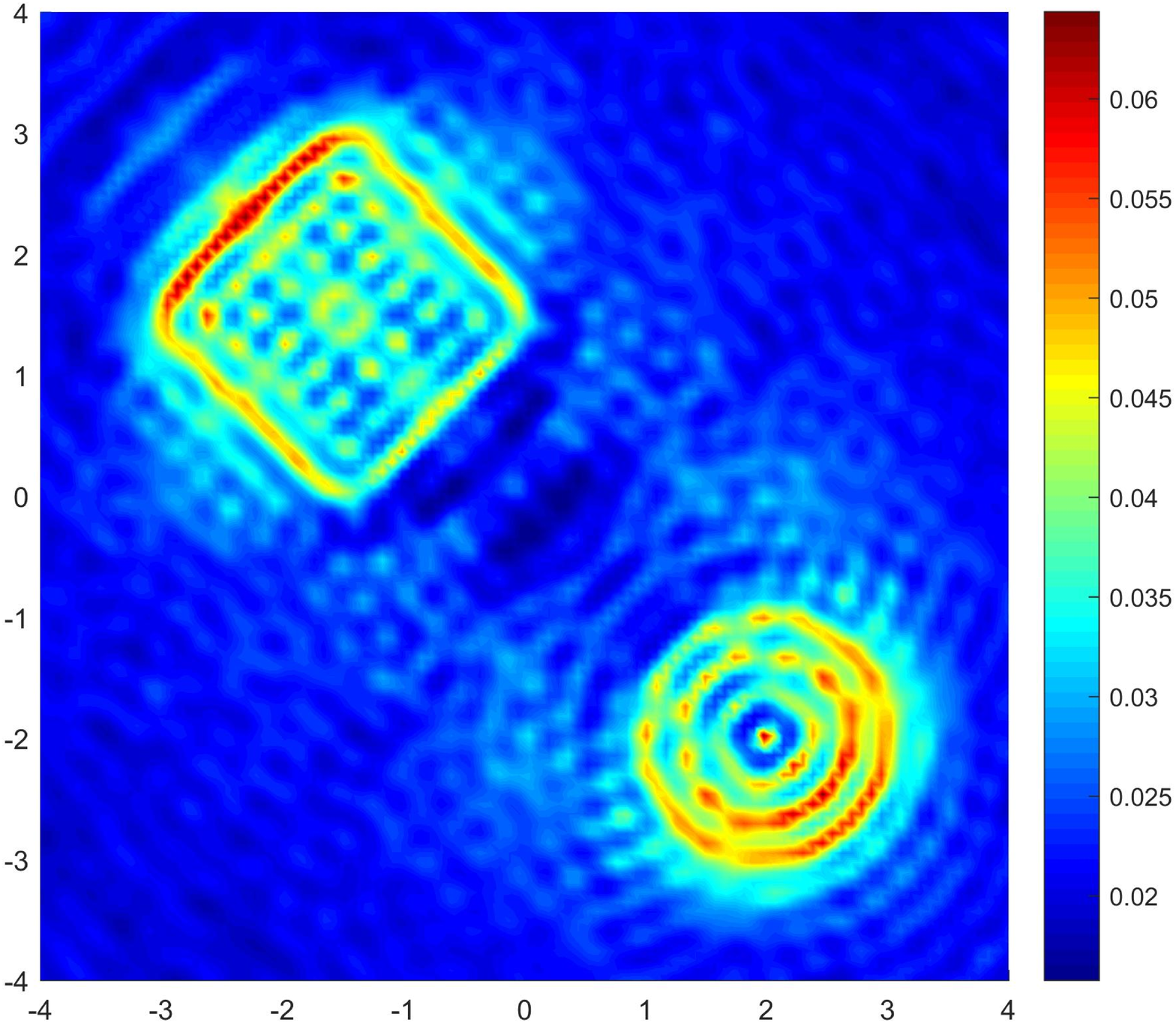}
%\caption{without \textbf{noise}}\label{fig2-ex1}
(b) {\color{HW}$k=10$, $R=15$, no noise}
\end{minipage}
\begin{minipage}[t]{0.4\linewidth}
\centering
\includegraphics[width=2.8in]{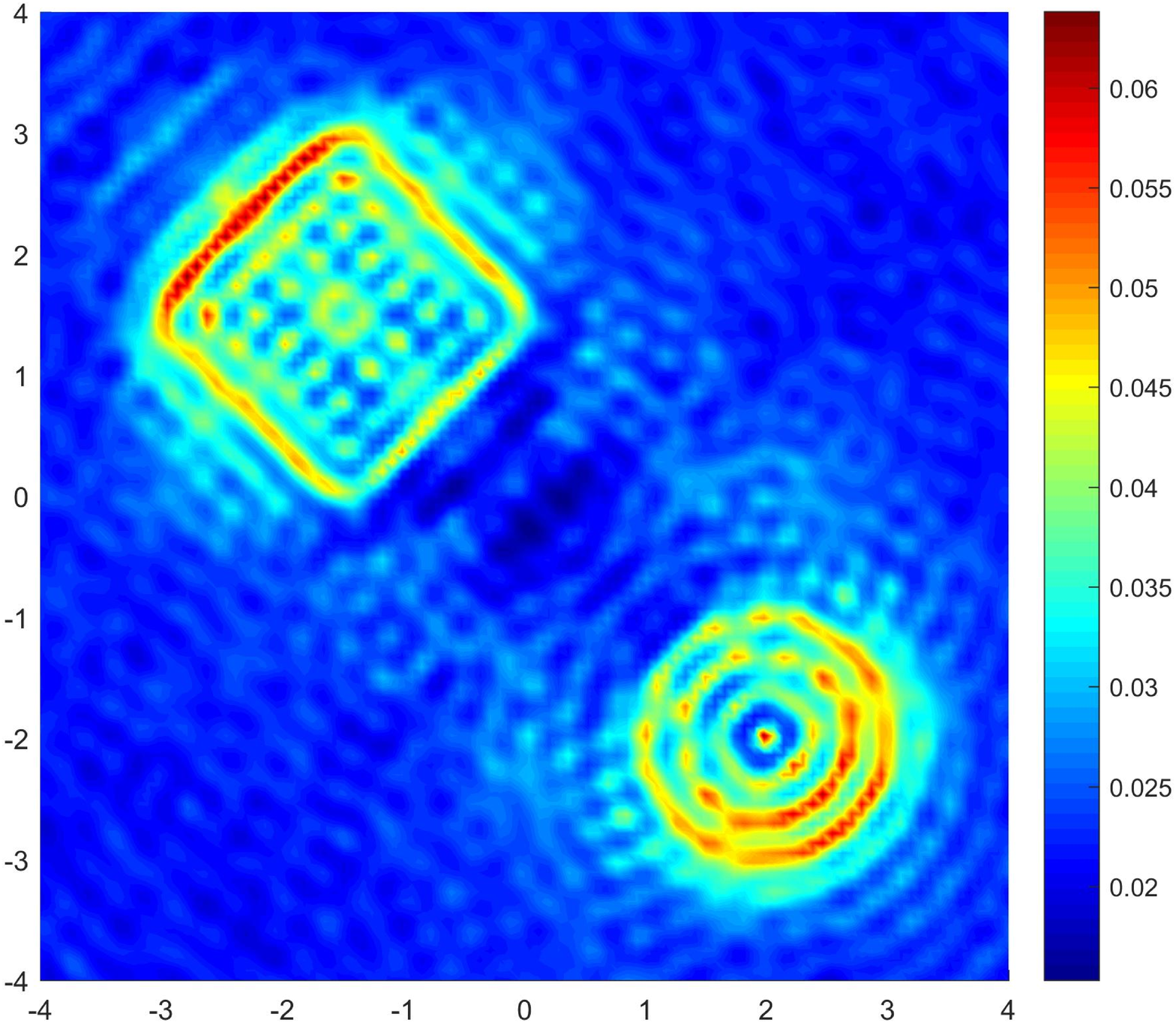}
%\caption{\textbf{2\% noise}}\label{fig3-ex1}
(c) {\color{HW}$k=10$, $R=15$, 10\% noise}
\end{minipage}\qquad\qquad
\begin{minipage}[t]{0.4\linewidth}
\centering
\includegraphics[width=2.8in]{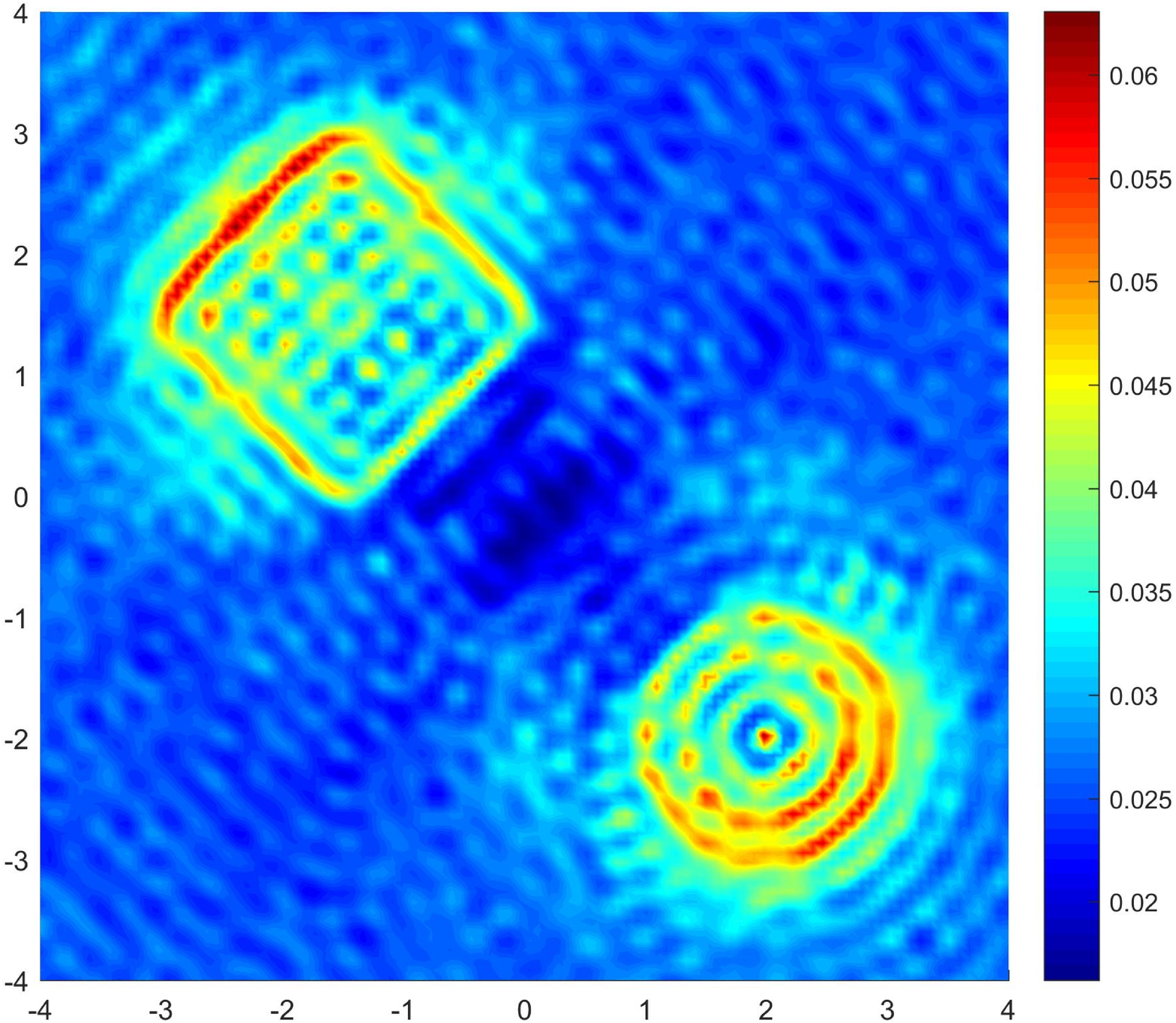}
%\caption{\textbf{5\% noise}}\label{fig4-ex1}
(d) {\color{HW}$k=10$, $R=15$, 20\% noise}
\end{minipage}
\caption{\color{HW}Reconstruction of a sound-soft obstacle with two disjoint components $D_1$ and $D_2$, where $D_1$ is
rounded square-shaped and $D_2$ is circle-shaped.
}\label{fig5}
\end{figure}

\section{Conclusion}\label{sec6}

In this paper, we considered the inverse scattering problem with phaseless {\color{HW}total-field} data at
a fixed frequency, associated with incident plane waves. An approximate factorization method is proposed to
reconstruct the unknown obstacles from phaseless {\color{HW}total-field} data measured on a circle of a
sufficiently large radius. The theoretical analysis of our approach is based on the asymptotic property in
the linear space $\mathcal{L}(H^{1/2}(\Sp^1), H^{-1/2}(\Sp^1))$
of the phaseless {\color{HW}total-field} operator defined in terms of the phaseless {\color{HW}total-field}
data measured on the circle of a large enough radius,
together with the factorization of the modified far-field operator.
Our inversion algorithm is independent of the physical properties of the unknown obstacles.
Numerical experiments indeed show that our inversion algorithm provides
satisfactory reconstruction results of the unknown obstacles.
Currently, we are extending this method to the case of incident point sources.
Moreover, it is interesting to study the more challenging case of inverse electromagnetic scattering problems.
This will be considered as a future work.

\section*{Acknowledgements}

This work is partly supported by the NNSF of China grants 11871466 and 91630309.
We thank the reviewers for their invaluable comments and suggestions which helped improve the presentation
of the paper.

%\newpage


\begin{thebibliography}{99}

\bibitem{AHN18} A.D. Agaltsov, T. Hohage and R.G. Novikov,
An iterative approach to monochromatic phaseless inverse scattering,
{\em Inverse Problems \bf35} (2018), 024001.

\bibitem{ACZ16} H. Ammari, Y. Chow and J. Zou,
Phased and phaseless domain reconstructions in the inverse scattering problem via scattering coefficients,
{\em SIAM J. Appl. Math. \bf76} (2016), 1000--1030.

\bibitem{AGH07} H. Ammari, R. Griesmaier and M. Hanke,
Identification of small inhomogeneities: asymptotic factorization,
{\em Math. Comp. \bf76} (2007), 1425--1448.

\bibitem{AG05}
T. Arens and N. Grinberg, A complete factorization method for scattering by periodic surfaces,
{\em Computing \bf75} (2005), 111--132.

\bibitem{BLL13} G. Bao, P. Li and J. Lv,
Numerical solution of an inverse diffraction grating problem from phaseless data,
{\em J. Opt. Soc. Am. A \bf30} (2013), 293--299.

\bibitem{BZ16} G. Bao and L. Zhang,
Shape reconstruction of the multi-scale rough surface from multi-frequency phaseless data,
{\em Inverse Problems \bf32} (2016), 085002.

\bibitem{BH13} Y. Boukari and H. Haddar,
The factorization method applied to cracks with impedance boundary conditions,
{\em Inverse Probl. Imaging \bf7} (2013), 1123--1138.

\bibitem{CC14} F. Cakoni and D. Colton,
{\em A Qualitative Approach to Inverse Scattering Theory}, Springer, Berlin, 2014.

\bibitem{CLS15} E.J. Cand\`{e}s, X. Li and M. Soltanolkotabi,
Phase retrieval via Wirtinger flow: theory and algorithms,
{\em IEEE Trans. Inform. Theory \bf61} (2015), 1985--2007.

\bibitem{CM08}
S.N. Chandler-Wilde and P. Monk, Wave-number-explicit bounds in time-harmonic scattering,
{\em SIAM J. Math. Anal. \bf39} (2008), 1428--1455.

\bibitem{Chen18} X. Chen, {\em Computational Methods for Electromagnetic Inverse Scattering},
Wiley, New York, 2018.

\bibitem{CHF2017} Z. Chen, S. Fang and G. Huang, A direct imaging method for the half-space inverse
scattering problems with phaseless data, {\em Inverse Problems Imaging \bf11} (2017), 901--916.

\bibitem{CH16} Z. Chen and G. Huang,
A direct imaging method for electromagnetic scattering data without phase information,
{\em SIAM J. Imaging Sci. \bf9} (2016), 1273--1297.

\bibitem{CH2017} Z. Chen and G. Huang, Phaseless imaging by reverse time migration: acoustic waves,
{\em Numer. Math. Theory Methods Appl. \bf10} (2017), 1--21.

\bibitem{CK83} D. Colton and R. Kress,
{\em Integral Equation Methods in Scattering Theory}, Wiley, New York, 1983.

\bibitem{CK11} D. Colton and R. Kress, Inverse scattering,
in: {\em Handbook of Mathematical Methods in Imaging} (O. Scherzer, ed.), Springer, New York, 2011,
pp. 551--598.

\bibitem{CK13} D. Colton and R. Kress,
{\em Inverse Acoustic and Electromagnetic Scattering Theory} (3rd Ed.), Springer, New York, 2013.

\bibitem{DLL19} H. Dong, J. Lai and P. Li,
Inverse obstacle scattering for elastic waves with phased or phaseless far-field data,
{\em SIAM J. Imaging Sciences \bf12} (2019), 809--838.

\bibitem{DZG19} H. Dong, D. Zhang and Y. Guo,
A reference ball based iterative algorithm for imaging acoustic obstacle from phaseless far-field data,
{\em Inverse Probl. Imaging \bf13} (2019), 177--195.

\bibitem{G08} R. Griesmaier,
An asymptotic factorization method for inverse electromagnetic scattering in layered media,
{\em SIAM J. Appl. Math. \bf68} (2008), 1378--1403.

\bibitem{GH09} R. Griesmaier and M. Hanke,
An asymptotic factorization method for inverse electromagnetic scattering in layered media II:
A numerical study, {\em Contemp. Math. \bf494} (2009), 61--79.

\bibitem{HYZZ14} G. Hu, J. Yang, B. Zhang and H. Zhang,
Near-field imaging of scattering obstacles with the factorization method,
{\em Inverse Problems \bf30} (2014), 095005.

\bibitem{I07} O. Ivanyshyn,
Shape reconstruction of acoustic obstacles from the modulus of the far field pattern,
{\em Inverse Probl. Imaging \bf1} (2007), 609--622.

\bibitem{IK10} O. Ivanyshyn and R. Kress,
Identification of sound-soft 3D obstacles from phaseless data,
{\em Inverse Probl. Imaging \bf4} (2010), 131--149.

\bibitem{IK11}
O. Ivanyshyn and R. Kress, Inverse scattering for surface impedance from phase-less far field data,
{\em J. Comput. Phys. \bf230} (2011), 3443--3452.

\bibitem{JLZ19c} X. Ji, X. Liu and B. Zhang,
Inverse acoustic scattering with phaseless far field data: Uniqueness, phase retrieval, and direct
sampling methods, {\em SIAM J. Imaging Sci. \bf12} (2019), 1163--1189.

\bibitem{JLZ19b} X. Ji, X. Liu and B. Zhang,
Phaseless inverse source scattering problem: Phase retrieval, uniqueness and direct sampling methods,
{\em J. Comput. Phys.: X\bf1} (2019), 100003.

\bibitem{JLZ19a} X. Ji, X. Liu and B. Zhang,
Target reconstruction with a reference point scatterer using phaseless far field patterns,
{\em SIAM J. Imaging Sci. \bf12} (2019), 372--391.

\bibitem{K98} A. Kirsch,
Characterization of the shape of a scattering obstacle using the spectral data of the far field operator,
{\em Inverse Problems \bf14} (1998), 1489--1512.

\bibitem{K11} A. Kirsch, {\em An Introduction to the Mathematical Theory of Inverse Problems} (2nd Ed.),
Springer, New York, 2011.

\bibitem{KG08} A. Kirsch and N. Grinberg,
{\em The Factorization Method for Inverse Problems}, Oxford Univ. Press, Oxford, 2008.

\bibitem{KL14} A. Kirsch and X. Liu,
A modification of the factorization method for the classical acoustic inverse scattering problems,
{\em Inverse Problems \bf30} (2014), 035013.

\bibitem{Kli14} M.V. Klibanov,
Phaseless inverse scattering problems in three dimensions,
{\em SIAM J. Appl. Math. \bf74} (2014), 392--410.

\bibitem{Kli17} M.V. Klibanov,
A phaseless inverse scattering problem for the 3-D Helmholtz equation,
{\em Inverse Probl. Imaging \bf11} (2017), 263--276.

\bibitem{KKNNBA18}
M.V. Klibanov, N.A. Koshev, D.-L. Nguyen, L.H. Nguyen, A. Brettin and V.N. Astratov, A numerical method
to solve a phaseless coefficient inverse problem from a single measurement of experimental data,
{\em SIAM J. Imaging Sci. \bf11} (2018), 2339--2367.

\bibitem{KNN19}
M.V. Klibanov, D.-L. Nguyen and L.H. Nguyen, A coefficient inverse problem with a single measurement of
phaseless scattering data, {\em SIAM J. Appl. Math. \bf79} (2019), 1--27.

\bibitem{KNP16}
M.V. Klibanov, L.H. Nguyen and K. Pan, Nanostructures imaging via numerical solution of a 3-{D} inverse
scattering problem without the phase information, {\em Appl. Numer. Math. \bf110} (2016), 190--203.

\bibitem{KR16} M.V. Klibanov and V.G. Romanov,
Reconstruction procedures for two inverse scattering problems without the phase information,
{\em SIAM J. Appl. Math. \bf76} (2016), 178--196.

\bibitem{K14} R. Kress, {\em Linear Integral Equations} (3rd Ed.), Springer, New York, 2014.

\bibitem{KR97} R. Kress and W. Rundell,
Inverse obstacle scattering with modulus of the far field pattern as data,
in: {\em Inverse Problems in Medical Imaging and Nondestructive Testing} (Oberwolfach, 1996),
Springer, Vienna, 1997, pp. 75--92.

\bibitem{L08} A. Lechleiter,
{\em Factorization Methods for Photonics and Rough Surfaces}, PhD thesis, Univ. Karlsruhe (TH),
Germany, 2008.

\bibitem{LL15} J. Li and H. Liu,
Recovering a polyhedral obstacle by a few backscattering measurements,
{\em J. Differ. Equations \bf259} (2015), 2101--2120.

\bibitem{LLW17} J. Li, H. Liu and Y. Wang,
Recovering an electromagnetic obstacle by a few phaseless backscattering measurements,
{\em Inverse Problems \bf33} (2017), 035011.

\bibitem{LZL09} L. Li, H. Zheng and F. Li,
Two-dimensional contrast source inversion method with phaseless data: TM case,
{\em IEEE Trans. Geosci. Remote Sens. \bf47} (2009), 1719--1736.

\bibitem{LZ10} X. Liu and B. Zhang,
Unique determination of a sound-soft ball by the modulus of a single far field datum,
{\em J. Math. Anal. Appl. \bf365} (2010), 619--624.

\bibitem{M76} A. Majda,
High frequency asymptotics for the scattering matrix and the inverse problem of acoustical scattering,
{\em Comm. Pure Appl. Math. \bf29} (1976), 261--291.

\bibitem{MDS93} M.H. Maleki and A.J. Devaney,
Phase-retrieval and intensity-only reconstruction algorithms for optical diffraction tomography,
{\em J. Opt. Soc. Am. A \bf10} (1993), 1086--1092.

\bibitem{MH17} S. Maretzke and T. Hohage,
Stability estimates for linearized near-field phase retrieval in X-ray phase contrast imaging,
{\em SIAM J. Appl. Math. \bf77} (2017), 384--408.

\bibitem{M00} W. McLean,
{\em Strongly Elliptic Systems and Boundary Integral Equations}, Cambridge Univ. Press, Cambridge, 2000.

\bibitem{MS19}
A. Moiola and E.A. Spence, Acoustic transmission problems: wavenumber-explicit bounds and resonance-free regions,
{\em Math. Models Methods Appl. Sci. \bf29} (2019), 317--354.

\bibitem{MNP16}
M. Moscoso, A. Novikov and G. Papanicolaou, Coherent imaging without phases,
{\em SIAM J. Imaging Sci. \bf9} (2016), 1689--1707.

\bibitem{MNPT17}
M. Moscoso, A. Novikov, G. Papanicolaou and C. Tsogka, Multifrequency interferometric imaging with
intensity-only measurements, {\em SIAM J. Imaging Sci. \bf10} (2017), 1005--1032.

\bibitem{NMP15}
A. Novikov, M. Moscoso and G. Papanicolaou, Illumination strategies for intensity-only imaging,
{\em SIAM J. Imaging Sci. \bf8} (2015), 1547--1573.

\bibitem{N15} R.G. Novikov,
Formulas for phase recovering from phaseless scattering data at fixed frequency,
{\em Bull. Sci. Math. \bf139} (2015), 923--936.

\bibitem{N16} R.G. Novikov,
Explicit formulas and global uniqueness for phaseless inverse scattering in multidimensions,
{\em J. Geom. Anal. \bf26} (2016), 346--359.

\bibitem{PZCY11} L. Pan, Y. Zhong, X. Chen and S.P. Yeo,
Subspace-based optimization method for inverse scattering problems utilizing phaseless data,
{\em IEEE Trans. Geosci. Remote Sens. \bf49} (2011), 981--987.

\bibitem{P06} R. Potthast, A survey on sampling and probe methods for inverse problems,
{\em Inverse Problems \bf22} (2006), R1--R47.

\bibitem{QYZ17} F. Qu, J. Yang and B. Zhang,
An approximate factorization method for inverse medium scattering with unknown buried objects,
{\em Inverse Problems \bf33} (2017), 035007.

\bibitem{QZ19} F. Qu and H. Zhang,
Locating a complex inhomogeneous medium with an approximate factorization method,
{\em Inverse Problems \bf35} (2019), 045001.

\bibitem{RY18} V.G. Romanov and M. Yamamoto, Phaseless inverse problems with interference waves,
{\em J. Inverse Ill-Posed Probl. \bf26} (2018), 681--688.

\bibitem{S14}
E.A. Spence, Wavenumber-explicit bounds in time-harmonic acoustic scattering,
{\em SIAM J. Math. Anal. \bf 46} (2014), 2987--3024.

\bibitem{XZZ18} X. Xu, B. Zhang and H. Zhang,
Uniqueness in inverse scattering problems with phaseless far-field data at a fixed frequency,
{\em SIAM J. Appl. Math. \bf78} (2018), 1737--1753.

\bibitem{XZZ18b} X. Xu, B. Zhang and H. Zhang,
Uniqueness in inverse scattering problems with phaseless far-field data at a fixed frequency. II,
{\em SIAM J. Appl. Math. \bf78} (2018), 3024--3039.

\bibitem{ZZ17b} B. Zhang and H. Zhang,
Imaging of locally rough surfaces from intensity-only far-field or near-field data,
{\em Inverse Problems \bf33} (2017), 055001.

\bibitem{ZZ17} B. Zhang and H. Zhang,
Recovering scattering obstacles by multi-frequency phaseless far-field data,
{\em J. Comput. Phys. \bf345} (2017), 58--73.

\bibitem{ZZ18} B. Zhang and H. Zhang,
Fast imaging of scattering obstacles from phaseless far-field measurements at a fixed frequency,
{\em Inverse Problems \bf34} (2018), 104005.

\bibitem{ZG18} D. Zhang and Y. Guo,
Uniqueness results on phaseless inverse acoustic scattering with a reference ball,
{\em Inverse Problems \bf34} (2018), 085002.

\bibitem{ZGLL18} D. Zhang, Y. Guo, J. Li and H. Liu,
Retrieval of acoustic sources from multi-frequency phaseless data,
{\em Inverse Problems \bf34} (2018), 094001.

\end{thebibliography}
\end{document}